\documentclass[12pt]{amsart}
\usepackage{upkg}

\title{Castelnuovo--Mumford regularity of skew-symmetric matrix Schubert varieties}
\author{Jack Chen-An Chou}
\address{Jack Chen-An Chou, School of Mathematics, University of Minnesota, Minneapolis, MN 55455, USA.}
\email{chou0188@umn.edu}

\begin{document}

\maketitle

\begin{abstract}
    Skew-symmetric matrix Schubert varieties are determinantal varieties obtained by intersecting matrix Schubert varieties with the space of skew-symmetric matrices. They are closely related to the orbit closures of the symplectic group action on the flag variety, and their torus-equivariant K-classes are the symplectic Grothendieck polynomials. We compute the Castelnuovo--Mumford regularity of skew-symmetric matrix Schubert varieties by giving a combinatorial formula for the degree of symplectic Grothendieck polynomials. In addition, we characterize the highest-degree homogeneous component of a symplectic Grothendieck polynomial and compute the maximal Castelnuovo--Mumford regularity of skew-symmetric matrix Schubert varieties.
\end{abstract}

\section{Introduction}
\label{s: intro}
The flag variety $\operatorname{Fl}_{n} = GL_n/B$ admits its classical stratification into Schubert cells and Schubert varieties, indexed by permutations. The corresponding cohomology and $K$-theory classes are represented by Schubert and Grothendieck polynomials. The Schubert varieties and their associated polynomials are central objects in the study of Schubert calculus. When $n$ is even, the symplectic group $\operatorname{Sp}_{n}$ also acts on $\operatorname{Fl}_{n}$ with finitely many orbits, indexed by fixed-point-free involutions. Polynomial representatives for the classes of these orbit closures were introduced by Wyser and Yong~\cite{Wyser13,WY17}, with further geometry and combinatorics developed by Hamaker, Marberg, and Pawlowski~\cite{HMPInvWordsI,MP20Ktheory,MP22Grobner,HMP22PD}.

Matrix Schubert varieties~\cite{Ful} and their skew-symmetric versions~\cite{WY17,HMP22PD,MP22Grobner} are generalized determinantal varieties that are closely related to orbit closures in $\operatorname{Fl}_{n}$. Since these varieties are Cohen--Macaulay~\cite{Ful,KM,Ram,MP22Grobner}, their Castelnuovo--Mumford regularity, a commutative-algebraic invariant that measures their algebro-geometric complexity, can be computed from the degree of their K-polynomials. The K-polynomial of the matrix Schubert variety $X_w$ is the Grothendieck polynomial $\fG_w(\x)$~\cite{KM}, and its degree has a combinatorial formula given by the Rajchgot index~\cite{psw24}. On the other hand, the K-polynomial of the skew-symmetric matrix Schubert variety $X^\mathrm{SS}_z$ is the symplectic Grothendieck polynomial $\fpfG_z(\x)$~\cite{MP22Grobner}. We compute the degree of $\fpfG_z(\x)$ by introducing a new
statistic on fixed-point-free involutions called the symplectic Rajchgot index $\sraj(\cdot)$ (see Definition~\ref{def:snow-sraj}).
\begin{thm}
\label{t: main_degree}
    Let $z$ be a fixed-point-free involution, then
    \begin{align*}
        \deg \fpfG_z (\x) = 2 \cdot \sraj(z).
    \end{align*}
\end{thm}
Our formula for the symplectic Rajchgot index is inspired by~\cite{py24}. Key combinatorial innovations of this paper include new notions of inversion code and Rothe diagram for fixed-point-free involutions called match code and match diagram (see Section~\ref{s: match code}) that induce natural recursions on various permutation statistics.

We prove Theorem~\ref{t: main_degree} by establishing a more precise statement. We can recursively construct a permutation from the match code of $z$ (see Definition~\ref{def:omega-recursion}) whose Grothendieck polynomial captures the highest-degree homogeneous component of $\fpfG_z(\x)$.
\begin{thm}
\label{t: main_top}
    Let $z$ be a fixed-point-free involution and let $w = \omega(\code(z))^{-1}$. Let $\widehat{\fG}_w(\x)$ and $\widehat{\mathfrak{G}}^{\fpf}_z(\x)$ denote the highest-degree homogeneous components of $\fG_w(\x)$ and $\fpfG_z(\x)$, respectively. Then
    \begin{align*}
       \widehat{\mathfrak{G}}^{\fpf}_z(\x) = (-1)^{\ell(w) - \ell_{\fpf}(z)}\widehat{\fG}_w(\x).
    \end{align*}
\end{thm}

An immediate corollary of Theorem~\ref{t: main_degree} is a combinatorial formula for the Castelnuovo--Mumford regularity of skew-symmetric matrix Schubert varieties.
\begin{cor}
\label{cor: Cm-reg}
Let $z$ be a fixed-point-free involution. Then 
\begin{align*}
    \reg(X^\mathrm{SS}_z) = \sraj(z) - \ell_{\fpf}(z).
\end{align*}
\end{cor}
To the author's knowledge, the best result previously was an upper bound for the regularity in the special case of vexillary skew-symmetric matrix Schubert varieties by Pechenik and St.~Denis~\cite{PS25}. Using Corollary~\ref{cor: Cm-reg}, we also determine the maximal Castelnuovo--Mumford regularity of skew-symmetric matrix Schubert varieties indexed by $\ifpf_{2m}$.
\begin{thm}
\label{t: maximize fpf reg}
For $m \in \ZZ_{>0}$, define $k$ by $\binom{k}{2} \leq m < \binom{k+1}{2}$. Then
    \begin{align*}
        \max_{z \in \ifpf_{2m}} \reg(X^{\mathrm{SS}}_z) =  2 \cdot \max_{w \in S_m} \reg(X_w) = m(m+1) - 2km + 2 \binom{k+1}{3}.
    \end{align*}
\end{thm}

The paper is organized as follows. Section~\ref{s:background} reviews the necessary
background. Section~\ref{s: match code} introduces match codes and match diagrams.
Section~\ref{s: paired_words and tight insertion} constructs the permutation $\omega(\code(z))$ and proves Theorem~\ref{t: main_top}. Section~\ref{sec:snow-degree} introduces the symplectic Rajchgot index and proves Theorem~\ref{t: main_degree} and Corollary~\ref{cor: Cm-reg}. Finally, Section~\ref{s: max cm} determines the maximal regularity and proves Theorem~\ref{t: maximize fpf reg}.

\section*{Acknowledgements}
We thank Joshua Arroyo, Zachary Hamaker, Eric Marberg, Brendan Pawlowski, and Anna Weigandt for helpful and inspiring conversations. We are grateful to Zachary Hamaker, Linus Setiabrata, and Anna Weigandt for carefully reading an earlier draft and providing useful feedback. We thank Zachary Hamaker, Eric Marberg, and Brendan Pawlowski for suggesting the problem.

Many of the lemmas and proofs in Sections~\ref{s: paired_words and tight insertion},~\ref{sec:snow-degree}, and~\ref{s: max cm} were obtained through extended conversations with GPT-5.6 Sol. In particular, the model proposed the key constructions in Definitions~\ref{d: tight insertion},~\ref{def:omega-recursion},~\ref{d: zv}, and~\ref{d :zs}. In addition, the model was used for grammatical checks and general streamlining of the paper.

\section{Background}
\label{s:background}
\subsection*{Permutations}
For a positive integer $n$, write $[n]=\{1,2,\ldots,n\}$ and let $S_n$ denote the symmetric group, with simple transpositions $s_i=(i,i+1)$. We consider permutations in both one-line notation and cycle notation. Certain permutations of even size will also be written as sequences of increasing consecutive pairs, called \definition{paired words} (see Definition~\ref{def: paired_word}). For $w \in S_n$, its \definition{inversion code}, denoted by $\invcode(w)$, is a weak composition of length $n$ whose $i\textsuperscript{th}$ entry records the number of indices $j > i$ such that $w(j) < w(i)$. The sum of the entries is the \definition{Coxeter length} of $w$, denoted as $\ell(w)$. The \definition{Rothe diagram} $D(w)$ of a permutation $w \in S_n$ is defined to be
\[
D(w) = \{(i,j)\in[n]\times[n]\colon i < w^{-1}(j) \textup{ and } j < w(i)\}.
\]
It is not hard to see that the number of cells in row $i$ of $D(w)$ is equal to $\invcode(w)_i$, and that the total number of cells in $D(w)$ is $\ell(w)$. 

Let $\ifpf_{n}$ be the set of fixed-point-free involutions in $S_n$ and $\ifpf=\bigcup_{n\geq0}\ifpf_{n}$. For $z\in\ifpf_{2m}$, we write its cycles in canonical order as
\[
    z=(b_1,c_1)(b_2,c_2)\cdots(b_m,c_m),
    \qquad b_i<c_i,
    \qquad b_1<b_2<\cdots<b_m.
\]
Hamaker, Marberg, and Pawlowski~\cite{HMPInvWordsI} defined analogues of Rothe diagram and inversion code for fixed-point-free involutions. The \definition{symplectic Rothe diagram} of  $z \in \ifpf_{n}$ is
\[
D^{\mathrm{Sp}}(z) = \{(i,j)\in [n] \times [n] : i < j < z(i), i < z(j)\}.
\]
Equivalently, $D^{\mathrm{Sp}}(z)$ is the subset of $D(z)$ that is strictly above the diagonal. In particular, our convention is the transpose of the definition in~\cite{HMPInvWordsI}. Since \(z\) is an involution, the off-diagonal part of \(D(z)\) is the disjoint union of \(D^{\mathrm{Sp}}(z)\) and its transpose. The \definition{fpf-involution code} of $z$ is a weak composition of length $n$ whose $i\textsuperscript{th}$ entry is
\[
c_{\mathrm{FPF},i}(z) = |\{j > i : z(j) < i \text{ and }z(j) < z(i)\}|,
\]
which is the number of cells in the $i\textsuperscript{th}$ column of $D^{\mathrm{Sp}}(z)$. The \definition{fpf-Coxeter length} $\ell_{\mathrm{FPF}}(z)$ is the sum of the entries in $c_{\mathrm{FPF}}(z)$. The inversion codes and fpf-involution codes are invariant under embedding the permutations/fixed-point-free involutions into a larger $S_n$/$\ifpf_{n}$ modulo appending zeros at the end.

In Section~\ref{s: match code}, we will introduce alternate analogues of Rothe diagrams and inversion codes for fixed-point-free involutions called match codes and match diagrams. They are different from the symplectic Rothe diagrams and the fpf-involution codes in general.

\begin{rem}
    Throughout this paper, the variables $y,z$ will always denote fixed-point-free involutions, while $u,w$ will denote permutations or paired words.
\end{rem}

Consider a permutation $w \in S_n$ in one-line notation. For $r \in [n]$, let $L_w(r)$ be the length of a longest increasing subsequence of $w$ whose first letter is $r$. We define
\[
    \mathcal J(w):=\sum_{r=1}^{n}L_w(r)
\]
to be their sum and, for $t \in \ZZ$, define
\[
    \rho_w(t):=\max\bigl(\{L_w(r):r>t\}\cup\{0\}\bigr).
\]
Equivalently, let $w|_{>t}$ denote the subword of $w$ obtained by deleting every letter weakly smaller than $t$, then $\rho_w(t)$ is the length of a longest increasing subsequence of $w|_{>t}$. Notice that $L_{w^{-1}}(i) = L_w(w(i))$, so summing over $i \in [n]$ gives $\mathcal{J}(w) = \mathcal{J}(w^{-1})$.

Pechenik, Speyer, and Weigandt~\cite{psw24}, in their study of highest-degree homogeneous components of Grothendieck polynomials, defined the \definition{Rajchgot code} for each $w \in S_n$ as the weak composition $\rajcode(w) = (r_1, \dots, r_n)$ where 
\begin{align*}
    r_i := n - i + 1 - L_{w^{-1}}(i) = n-i +1 - L_w(w(i)).
\end{align*}
Equivalently, $r_i$ is the number of letters skipped by any longest increasing subsequence starting at position $i$ in the one-line notation of $w$. They further define the sum of its entries as the \definition{Rajchgot index} on permutations.
\begin{align*}
    \raj(w) := \sum_{i=1}^n r_i =  \binom{n+1}{2} - \mathcal{J}(w^{-1}) = \binom{n+1}{2} - \mathcal{J}(w).
\end{align*}

\subsection*{Grothendieck polynomials}
For $w \in S_n$, Lascoux and Sch\"utzenberger \cite{LS:Groth} introduced the Grothendieck polynomial $\fG_w(\x)$ as the polynomial representative of the structure-sheaf class of the corresponding Schubert variety in the Grothendieck ring of the complete flag variety. In general, the Grothendieck polynomials are not homogeneous.  The lowest-degree homogeneous components of the Grothendieck polynomials are the Schubert polynomials, which represent the cohomology classes of Schubert varieties in flag varieties.

Let $s_i$ act on $\ZZ[\x]$ by interchanging $x_i$ and $x_{i+1}$.  Define
\[
    \partial_i f=\frac{f-s_if}{x_i-x_{i+1}}
    \qquad\text{and}\qquad
    \overline{\partial_i} f
    =\partial_i\bigl((1 - x_{i+1})f\bigr).
\]
The set of \definition{Grothendieck polynomials} is the unique family of polynomials in $\ZZ[\x]$ satisfying
\[
\fG_w(\x) = \begin{cases}
x_1^{n-1}x_2^{n-2}\cdots x_{n-1} & \text{if } w = n(n-1)\cdots1, \\
\overline{\partial_i} \fG_{ws_i}(\x) & \text{if } w(i) < w(i+1).
\end{cases}
\]
They are stable under adjoining a final fixed point and have coefficients in $\ZZ$.

Pechenik, Speyer, and Weigandt showed that the degrees of Grothendieck polynomials are given by the $\raj(\cdot)$ statistic, and that $\raj(\cdot)$ is invariant under taking inverses. 
\begin{thm}~\cite{psw24}*{Theorem~1.1 and Corollary~4.5}
\label{t: psw raj}
    For $w \in S_n$, 
    \begin{align*}
        \deg \fG_w(\x) = \raj(w) = \raj(w^{-1}) =  \deg \fG_{w^{-1}}(\x).
    \end{align*}
\end{thm}
In addition to their formula, the climbing chain models introduced by
Dreyer, Mészáros, and St. Dizier~\cite{DMS}, and the snow diagrams introduced by Pan and Yu~\cite{py24} also compute $\raj(\cdot)$ combinatorially. Our formula for $\sraj(\cdot)$ resembles that of Pan and Yu's snow diagrams (see Section~\ref{sec:snow-degree}).

The Rajchgot code is also closely related to the characterization of highest-degree homogeneous components of Grothendieck polynomials.
\begin{thm}~\cite{psw24}*{Theorem~1.1 and Theorem~1.4}
\label{t: rajcode distinguish}
    For $w \in S_n$, the leading term of $\widehat{\fG}_w(\x)$ for every term order satisfying $x_1 < \cdots < x_n$ is a scalar multiple of the monomial $x^{\rajcode(w)}$. Moreover, two highest-degree homogeneous components \(\widehat{\mathfrak G}_u(\x)\) and \(\widehat{\mathfrak G}_v(\x)\) agree up to a scalar if and only if $\rajcode(u)=\rajcode(v)$.
\end{thm}
In particular, the highest-degree homogeneous component of a Grothendieck polynomial is completely determined, up to scalar multiplication, by its Rajchgot code.

\subsection*{Symplectic Grothendieck polynomials}
\label{s : symplectic grothendieck}
For even $n$, the symplectic group $\operatorname{Sp}_n$ acts on the complete flag variety $\operatorname{Fl}_n$ with finitely many orbits indexed by $\ifpf_n$. Wyser and Yong~\cite{WY17} constructed polynomial representatives for the cohomology and $K$-theory classes of the closures of these orbits. Marberg and Pawlowski~\cite{MP20Ktheory} further studied these $K$-theory representatives and characterized their stable versions, called the symplectic Grothendieck polynomials.

The set of \definition{symplectic Grothendieck polynomials} is the unique family of polynomials in $\ZZ[\x]$ satisfying
\begin{equation*}
\fpfG_z(\x) = 
\begin{dcases}
    \prod_{1 \leq i < j \leq n-i}(x_i + x_j - x_ix_j) & \text{if } z = n(n-1) \cdots 1, \\
    \overline{\partial_i} \fpfG_{s_i z s_i}(\x) & \text{if } i + 1 \neq z(i) < z(i+1) \neq i.
\end{dcases}
\end{equation*}
The lowest degree homogeneous part of $\fpfG_z(\x)$ is the \definition{symplectic Schubert polynomial} $\fpfS_z(\x)$, with degree equal to $\ell_{\fpf}(z)$.

\begin{rem}
    Grothendieck polynomials and symplectic Grothendieck polynomials often appear in the literature with an extra $\beta$ parameter. In this paper, we only consider the specialization $\beta = -1$.
\end{rem}

Marberg and Pawlowski defined a special class of fixed-point-free involutions for which the symplectic Grothendieck polynomials are relatively easy to compute. 
\begin{defn}
\label{d : sp-dominant}
We say that $z \in \ifpf_{n}$ is \definition{$\mathrm{Sp}$-dominant} if there exists a strict partition $ \mu = (\mu_1 > \mu_2 > \cdots > \mu_k > 0)$ such that
\begin{align*}
    D^{\mathrm{Sp}}(z) = \{(j, i + j)\in [k] \times [n] : 1  \leq i \leq \mu_j \}.
\end{align*}
\end{defn}
In this case, the symplectic Grothendieck polynomials can be computed directly from the symplectic Rothe diagrams.
\begin{thm}~\cite{MP20Ktheory}*{Theorem~3.8}
    Let $z \in \ifpf_{n}$ be $\mathrm{Sp}$-dominant, then
    \begin{align*}
        \fpfG_z(\x) = \prod_{(i,j) \in D^{\mathrm{Sp}}(z)} (x_i + x_j - x_ix_j).
    \end{align*}
\end{thm}

For each $z \in \ifpf_{n}$, Marberg and Pawlowski defined the set $\mathcal B_{\mathrm{FPF}}(z)$ of \definition{fixed-point-free Hecke atoms} and showed that it is the set of permutations that appear in the Grothendieck expansion of $\fpfG_z(\x)$.
\begin{thm}~\cite{MP20Ktheory}*{Theorem~3.12}
\label{t: symplectic to Gro expansion}
     For $z\in\ifpf_n$,
\[
    \fpfG_z(\x) = \sum_{w\in\mathcal B_{\mathrm{FPF}}(z)}
    (-1)^{\ell(w)-\ell_{\mathrm{FPF}}(z)}\G_w(\x).
\]
\end{thm}
For combinatorial reasons, it is often more convenient to consider the set of inverses of fixed-point-free Hecke atoms. Let $\PP(z) = \{w : w^{-1} \in \mathcal B_{\mathrm{FPF}}(z)\}$ be the set of \definition{inverse Hecke atoms}. For an explicit definition of $\PP(z)$, see Proposition~\ref{p:paired-class-inverse-atoms}.

\subsection*{Castelnuovo--Mumford regularity}
We first recall the commutative-algebraic definition of
Castelnuovo--Mumford regularity.
Let $ R=\mathbb C[x_1,\ldots,x_N]$ be a standard-graded polynomial ring, so that $\deg(x_i)=1$, and let $I\subseteq R$ be a homogeneous ideal. For $d\in \ZZ$, write $R(-d)$ for the graded free $R$-module obtained from $R$ by shifting all degrees by $d$.

The \definition{minimal graded free resolution} of \(R/I\), which is unique up to isomorphism, is an exact sequence of graded $R$-module maps of the form
\[
    0\longrightarrow
    \bigoplus_{d\in\mathbb Z}R(-d)^{\beta_{p,d}(R/I)}
    \longrightarrow\cdots\longrightarrow
    \bigoplus_{d\in\mathbb Z}R(-d)^{\beta_{0,d}(R/I)}
    \longrightarrow R/I\longrightarrow0.
\]
The integers $\beta_{i,d}(R/I)$ occurring in this resolution are the graded Betti numbers of $R/I$.  The \definition{Castelnuovo--Mumford regularity} of $R/I$ is
\begin{equation*}
    \reg(R/I)
    :=
    \max\{d-i:\beta_{i,d}(R/I)\neq0\}.
\end{equation*}
When $X=\operatorname{Spec}(R/I)$ is an affine variety, we also write
$\reg(X)$ for $\reg(R/I)$.

The Hilbert series of $R/I$ is
\[
    \hilb(R/I;t) := \sum_{d\geq0}\dim_{\CC}(R/I)_d\,t^d.
\]
It has a unique expression of the form
\begin{equation*}
    \hilb(R/I;t) = \frac{K(R/I;t)}{(1-t)^N},
\end{equation*}
where $K(R/I;t)\in \ZZ[t]$.  The numerator
$K(R/I;t)$ is the \definition{$K$-polynomial} of $R/I$.

If $\prec$ is a term order, then $R/I$ and
$R/\operatorname{in}_{\prec}(I)$ have the same Hilbert series and hence
the same $K$-polynomial.  This permits the use of Gr\"obner degenerations
to study $K(R/I;t)$.  

Recall that $R/I$ is \definition{Cohen--Macaulay} if
\[
    \operatorname{depth}(R/I)=\dim(R/I).
\]
When $R/I$ is Cohen--Macaulay, $\reg(R/I)$ is determined by the degree of the K-polynomial. The following lemma is well known to experts (see~\cite{BV15}*{Lemma~2.5} and~\cite{RRRSW}).

\begin{lem}~\cite{psw24}*{Lemma~2.2}~\cite{PS25}*{Lemma~2.8}
\label{lem:CM-regularity-K-polynomial}
If $R/I$ is Cohen--Macaulay, then
\[
    \reg(R/I)
    =
    \deg_t K(R/I;t)-\operatorname{ht}(I).
\]
\end{lem}

When $R/I$ is not Cohen--Macaulay, regularity itself is not preserved under Gr\"obner degeneration in general.

\subsection*{Skew-symmetric matrix Schubert varieties}
Let
\[
    \operatorname{Mat}^{\mathrm{SS}}_n
    :=
    \{A\in\operatorname{Mat}_{n\times n}(\mathbb C):
      A^{\mathsf T}=-A\}
\]
be the set of skew-symmetric matrices.
Its coordinate ring is
\[
    R = R_n^{\mathrm{SS}}
    :=
    \mathbb C[\operatorname{Mat}^{\mathrm{SS}}_n]
    =
    \mathbb C[u_{ij}:n\geq i>j\geq1].
\]

For $z\in\ifpf_n$, let 
\begin{align*}
    X_z^{\mathrm{SS}} = \left\{ A\in\operatorname{Mat}^{\mathrm{SS}}_n: \operatorname{rank}A_{[i][j]} \leq \operatorname{rank}z_{[i][j]} \text{ for all }i,j \right\}
\end{align*}
be the \definition{skew-symmetric matrix Schubert variety} of $z$, and let
\[
    I_z^{\mathrm{SS}}:= I(X_z^{\mathrm{SS}}) \subseteq R_n^{\mathrm{SS}}.
\]
Marberg and Pawlowski proved that a natural initial ideal of
$ I_z^{\mathrm{SS}}$ is squarefree and is the Stanley--Reisner ideal of a
shellable simplicial complex~\cite[Theorem~1.6]{MP22Grobner}. It follows that the coordinate ring
\[
    A_z := \mathbb C[X_z^{\mathrm{SS}}] = R/ I_z^{\mathrm{SS}}
\]
is Cohen--Macaulay. Their primary decomposition~\cite[Theorem~1.5]{MP22Grobner} also shows that 
\begin{equation*}
    \operatorname{ht}( I_z^{\mathrm{SS}})
    = \operatorname{codim}(X_z^{\mathrm{SS}})
    = \ell_{\mathrm{FPF}}(z).
\end{equation*}

Therefore, to compute the Castelnuovo--Mumford regularity of skew-symmetric matrix Schubert varieties, it remains to study the degree of the K-polynomial of $R/ I_z^{\mathrm{SS}}$. Marberg and Pawlowski worked with a finer torus multigrading on \(R\). We compare its \(K\)-polynomial with the \(K\)-polynomial for the standard grading. Let $K_z^{\mathrm{std}}(t) := K(A_z;t)$ denote the $K$-polynomial in the standard grading $\deg(u_{ij})=1$. Let $K^T_z(\mathbf a) := K(A_z;\mathbf a)$ denote the $K$-polynomial with respect to the multigrading $\deg_T(u_{ij})=e_i+e_j$, recorded multiplicatively as the monomial $a_i a_j$. The Hilbert series under the torus multigrading is 
\begin{align*}
    \hilb_T(A_z; \mathbf a) := \sum_{\alpha \in \ZZ^n} \dim_{\CC}(A_z)_{\alpha}a^{\alpha} = \frac{K^T_z(\mathbf a)}{\prod_{i>j}(1-a_i a_j)}.
\end{align*}

\begin{rem}
Marberg and Pawlowski denoted the K-polynomial of $R/I$ by $\mathcal K(I)$ in~\cite{MP22Grobner}*{Definition~4.16}. We use a different notation and denote the same object as $K(R/I)$.
\end{rem}

\begin{lem}
\label{lem:torus-standard-specialization}
For $z\in\ifpf_n$,
\[
    \operatorname{Hilb}_T(A_z;s,\ldots,s)
    =\operatorname{Hilb}_{\mathrm{std}}(A_z;s^2),
    \quad \text{ and }\quad
    K_z^T(s,\ldots,s)=K_z^{\mathrm{std}}(s^2).
\]
\end{lem}

\begin{proof}
A monomial $\prod_{i>j}u_{ij}^{b_{ij}}$ of multidegree $\alpha$
satisfies $|\alpha|=2\sum_{i>j}b_{ij}$. Thus its standard degree is $|\alpha|/2$. Since $ I_z^{\mathrm{SS}}$ is homogeneous with respect to the multigrading, the same compatibility holds in the coordinate ring $A_z = R/ I^{\mathrm{SS}}_z$. Let $ (A_z)_d^{\mathrm{std}}$ denote the degree $d$ part of $A_z$ with respect to the standard grading and let $(A_z)_\alpha^T$ be the degree $\alpha$ part of $A_z$ with respect to the torus equivariant grading. Then, as vector spaces, 
\[
    (A_z)_d^{\mathrm{std}}
    =\bigoplus_{|\alpha|=2d}(A_z)_\alpha^T.
\]
Since only finitely many summands occur for each $d$, we have
\[
\begin{aligned}
    \operatorname{Hilb}_T(A_z;s,\ldots,s) =\sum_\alpha\dim_{\mathbb C}(A_z)_\alpha^T\,s^{|\alpha|} =\sum_{d\geq0}\dim_{\mathbb C}(A_z)_d^{\mathrm{std}}\,s^{2d} =\operatorname{Hilb}_{\mathrm{std}}(A_z;s^2).
\end{aligned}
\]
The denominators of the Hilbert series are $\prod_{i>j}(1-a_i a_j)$ and $(1-t)^{\binom{n}{2}}$, respectively. Therefore, setting every $a_i \mapsto s$  and $t \mapsto s^2$ gives
\[
    \frac{K_z^T(s,\ldots,s)}{(1-s^2)^{\binom{n}{2}}}
    =\frac{K_z^{\mathrm{std}}(s^2)}{(1-s^2)^{\binom{n}{2}}}. \qedhere
\]
\end{proof}

Marberg and Pawlowski called $K_z^T(\mathbf a)$ the symplectic
Grothendieck polynomial in character variables
\cite[Definition~4.19]{MP22Grobner}. The symplectic Grothendieck polynomials in the $\x$-variables defined earlier are obtained by the change of variables $a_i \mapsto 1 - x_i$:
\[
    \fpfG_z(\mathbf x)=K_z^T(1-x_1,\ldots,1-x_n).
\]
This substitution gives what is known as the \definition{twisted K-polynomial}~\cite{CCRMM24}.

To relate the degrees of $\fpfG_z(\mathbf x)$, $K_z^{\mathrm{std}}(t)$, and $K_z^T(\mathbf a)$, we require the following formula.

\begin{thm}~\cite{MP22Grobner}*{Theorem~4.28}
\label{t:MP-K-polynomial-expansion}
Let $z\in\ifpf_n$, then the $K$-polynomial of
$X_z^{\mathrm{SS}}$ is
\begin{align}
\label{eq:MP-K-polynomial-expansion}
    K_z^T(\mathbf a)
    =
    \sum_{D\in\mathcal D(z)}
    (-1)^{|D|-\ell_{\mathrm{FPF}}(z)}
    \prod_{(i,j)\in D}(1-a_i a_j),
\end{align}
where $\mathcal D(z)$ is the set of extended fpf-involution pipe dreams of $z$. 
\end{thm}
We do not need the precise definition of extended fpf-involution pipe dreams. It is only important to us that $\mathcal D(z)$ is a nonempty, finite collection of subsets of $[n] \times [n]$.

\begin{lem}
\label{lem:standard-torus-degree-factor}
For $z\in\ifpf_n$,
\[
    \deg_{\mathbf x}\fpfG_z(\mathbf x)
    =\deg_{\mathbf a}K_z^T(\mathbf a)
    =2\deg_t K_z^{\mathrm{std}}(t).
\]
\end{lem}

\begin{proof}
Let $M=\max_{D\in\mathcal D(z)}|D|$. By Theorem~\ref{t:MP-K-polynomial-expansion}, the change of variables $a_i \mapsto 1 - x_i$ gives 
\begin{align*}
    \fpfG_z(\mathbf x) = K_z^T(1-x_1,\ldots,1-x_n) = \sum_{D\in\mathcal D(z)}
    (-1)^{|D|-\ell_{\mathrm{FPF}}(z)}
    \prod_{(i,j)\in D}(x_i + x_j - x_ix_j).
\end{align*}
The highest-degree homogeneous component is 
\begin{align*}
    (-1)^{\ell_{\mathrm{FPF}}(z)}
    \sum_{\substack{D\in\mathcal D(z)\\|D|=M}}
       \prod_{(i,j)\in D}x_i x_j.
\end{align*}
In particular, it is nonzero and all nonzero coefficients have the same sign, so this change of variables does not induce any cancellation in the highest-degree homogeneous part. Therefore, $\deg_{\x} \fpfG_z(\mathbf x) = 2M = \deg_{\mathbf a} K_z^T(\mathbf a)$.

To see the second equality, notice that by Theorem~\ref{t:MP-K-polynomial-expansion}, the highest-degree homogeneous component of $K_z^T(\mathbf a)$ is
\[
    (-1)^{\ell_{\mathrm{FPF}}(z)}
    \sum_{\substack{D\in\mathcal D(z)\\|D|=M}}
       \prod_{(i,j)\in D}a_i a_j.
\]
Again, it is nonzero and all nonzero coefficients have the same sign. So setting $a_1=\cdots=a_n=s$ does not induce any cancellation, hence by Lemma~\ref{lem:torus-standard-specialization},
\[
 \deg_{\mathbf a}K_z^T(\mathbf a) =\deg_s K_z^T(s,\ldots,s) =\deg_s K_z^{\mathrm{std}}(s^2) =2\deg_t K_z^{\mathrm{std}}(t). \qedhere
\]
\end{proof}

Therefore, by Lemma~\ref{lem:CM-regularity-K-polynomial} and Lemma~\ref{lem:standard-torus-degree-factor}, computing the Castelnuovo–Mumford regularity of skew-symmetric matrix Schubert varieties reduces to computing the degree of symplectic Grothendieck polynomials.

\section{Match code and match diagram}
\label{s: match code}

In this section, we introduce the match code and match diagram of $z \in \ifpf_{2m}$. They serve as analogues of inversion codes and Rothe diagrams for permutations. In general, these differ from the fpf-involution codes and symplectic Rothe diagram defined by Hamaker, Marberg, and Pawlowski~\cite{HMP22PD}.

\begin{defn}
\label{d: match diagram}
     For $ z \in \ifpf_{2m}$, define the \definition{match diagram} of $z = (b_1, c_1)(b_2,c_2) \cdots (b_m,c_m)$ to be the following subset of $[m] \times [2m]$:
    \begin{align*}
        \MD(z) = \{(i,b_j)\, : \,b_i < b_j < c_i, i < j \} \cup \{(i,c_j)\, : \,b_i < c_j < c_i, i < j \}.
    \end{align*}
    Equivalently, $\MD(z)$ can be created by placing lasers at $(i,b_i)$ and $(i,c_i)$. Lasers at  $(i,b_i)$ emit beams downward and lasers at $(i,c_i)$ emit beams downward and to the right. The cells in $[m] \times [2m]$ that are not hit by any laser beams form the match diagram $\MD(z)$.
\end{defn}
\begin{defn}
    \label{d: match_code}
        For $ z \in \ifpf_{2m}$, the \definition{match code} of $z$, denoted as $\code(z)$, is the weak composition of length $m$ whose $i\textsuperscript{th}$ entry is the number of cells in row $i$ of $\MD(z)$.
\end{defn}

\begin{exa}
    Let $z = (1,5)(2,3)(4,8)(6,7)$. We place lasers at $(1,1), (1,5)$ since the first cycle has numbers $1$ and $5$. In each row, the left laser shoots downwards and the right laser shoots downwards and to the right. The remaining cells form $\MD(z)$.
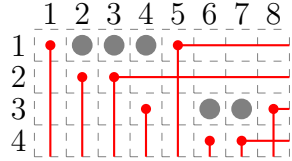
    \begin{align*}
        \begin{tikzpicture}
            [x=1em,y=1em,thick,color = blue]
            \draw[step=1,gray,ultra thin,dashed] (0,0) grid (8,4);
            \node[color=black] at (-0.5,3.5) {$1$};
            \node[color=black] at (-0.5,2.5) {$2$};
            \node[color=black] at (-0.5,1.5) {$3$};
            \node[color=black] at (-0.5,0.5) {$4$};
            \node[color=black] at (0.5,4.5) {$1$};
            \node[color=black] at (1.5,4.5) {$2$};
            \node[color=black] at (2.5,4.5) {$3$};
            \node[color=black] at (3.5,4.5) {$4$};
            \node[color=black] at (4.5,4.5) {$5$};
            \node[color=black] at (5.5,4.5) {$6$};
            \node[color=black] at (6.5,4.5) {$7$};
            \node[color=black] at (7.5,4.5) {$8$};
            \filldraw [red] (0.5,3.5) circle (1.5pt);
            \filldraw [red] (4.5,3.5) circle (1.5pt);
            \filldraw [red] (1.5,2.5) circle (1.5pt);
            \filldraw [red] (2.5,2.5) circle (1.5pt);
            \filldraw [red] (3.5,1.5) circle (1.5pt);
            \filldraw [red] (7.5,1.5) circle (1.5pt);
            \filldraw [red] (5.5,0.5) circle (1.5pt);
            \filldraw [red] (6.5,0.5) circle (1.5pt);
            \draw[color=red] (0.5,3.5)--(0.5, 0);
            \draw[color=red] (4.5,3.5)--(4.5, 0);
            \draw[color=red] (4.5,3.5)--(8, 3.5);
            \draw[color=red] (1.5,2.5)--(1.5, 0);
            \draw[color=red] (2.5,2.5)--(2.5, 0);
            \draw[color=red] (2.5,2.5)--(8, 2.5);
            \draw[color=red] (3.5,1.5)--(3.5, 0);
            \draw[color=red] (7.5,1.5)--(7.5, 0);
            \draw[color=red] (7.5,1.5)--(8, 1.5);
            \draw[color=red] (5.5,0.5)--(5.5,0);
            \draw[color=red] (6.5,0.5)--(6.5,0);
            \draw[color=red] (6.5,0.5)--(8,0.5);
            \filldraw [gray] (1.5,3.5) circle (3.5pt);
            \filldraw [gray] (2.5,3.5) circle (3.5pt);
            \filldraw [gray] (3.5,3.5) circle (3.5pt);
            \filldraw [gray] (5.5,1.5) circle (3.5pt);
            \filldraw [gray] (6.5,1.5) circle (3.5pt);
        \end{tikzpicture}
    \end{align*}
    Counting the number of cells in each row gives $\code(z) = (3,0,2,0)$.
\end{exa}

\begin{lem}
\label{l: two diagram bijection}
For $z\in\ifpf_{2m}$, the map 
\begin{align*}
    \varphi:\,\MD(z)&\longrightarrow D^{\mathrm{Sp}}(z)\\
     (i,r)&\longmapsto (b_i,r)
\end{align*}
is a bijection.
\end{lem}

\begin{proof}
Suppose $(i,r)\in\MD(z)$, then $r$ is an endpoint of a cycle $(b_j,c_j)$ with $j>i$, and $b_i<r<c_i=z(b_i)$. Since $j>i$, we have $b_j>b_i$. So in either case, $z(r)>b_i$. Therefore $ b_i<r<z(b_i)$ and $  b_i<z(r)$ so $ (b_i,r)\in D^{\mathrm{Sp}}(z)$.

Conversely, suppose $(p,r)\in D^{\mathrm{Sp}}(z)$. Since $ p<z(p)$, the letter $p$ is the left endpoint of $(b_i,c_i)$ for some $i$. The definition of $D^{\mathrm{Sp}}(z)$ implies that $b_i<r<c_i$ and $ b_i<z(r)$. Thus both $r$ and $z(r)$ are greater than $b_i$. Hence the cycle containing $r$ has left endpoint greater than $b_i$, so it is a
later cycle $(b_j,c_j)$ with $j>i$. Therefore $(i,r)\in\MD(z)$.
\end{proof}

Similar to the fpf-involution codes, the match codes sum up to the fpf-Coxeter length of the fixed-point-free involution.
\begin{cor}
    For $z\in\ifpf_{2m}$, 
    \begin{align*}
         \sum_{i=1}^m \code_i(z)=\ell_{\mathrm{FPF}}(z).
    \end{align*}
\end{cor}
\begin{proof}
    By Lemma~\ref{l: two diagram bijection},
    \begin{align*}
        &\sum_{i=1}^m \code_i(z) = |\MD(z)| = |D^{\mathrm{Sp}}(z)| = \ell_{\mathrm{FPF}}(z). \qedhere
    \end{align*}
\end{proof}

We say a weak composition $\alpha$ of length $m$ is \definition{valid} if $\alpha_i \leq 2(m-i)$ for all $i \in [m]$. 

\begin{lem}
\label{l: code composition bijection}
    Match code gives a bijection between $\ifpf_{2m}$ and valid weak compositions of length $m$.
\end{lem}
\begin{proof}
    It is clear from definition that given $z \in \ifpf_{2m}$, $\code(z)$ is a valid weak composition. We construct the inverse map from valid weak compositions $\alpha$ to fixed-point-free involutions. To construct cycle $i$, let $b_i$ be the smallest unused number and let $c_i$ be the $(\alpha_i + 2)$ smallest unused number. Exactly \(\alpha_i\) unused labels lie between \(b_i\) and \(c_i\), and these become endpoints of later cycles. Since $\alpha_i \leq 2(m-i)$, the first $m$ cycles must use the numbers $[2m]$, so $z \in \ifpf_{2m}$.
\end{proof}

Just like the inversion code for permutations, the match code for a fixed-point-free involution $z \in \ifpf_{2m}$ is invariant under embedding $z$ into $\ifpf_{2m+2}$.
After identifying \(z\in\mathcal I^{\mathrm{FPF}}_{2m}\) with the involution obtained by adjoining \((2m+1,2m+2)\), match code induces a bijection with finite weak compositions modulo trailing zeros.

\begin{defn}
    We now define a recurrence on fixed-point-free involutions using their match codes. For $z \in \ifpf_{2m}$, removing the first entry of $\code(z)$ gives another weak composition $\alpha$ such that $\alpha_i \leq 2((m-1)-i)$, so $\alpha = \code(y)$ for some $ y \in \ifpf_{2m-2}$. We write $z = (a,y)$ if $a$ is the first entry of $\code(z)$. By Lemma~\ref{l: code composition bijection}, $z$ uniquely determines $a$ and $y$ and vice versa. 
\end{defn}
We will be using this recurrence extensively for the rest of the paper alongside the following strictly increasing map: for $a\geq0$, let $\iota_a$ be the map
\[
    \iota_a(r) :=
    \begin{cases}
        r+1,&r\leq a,\\
        r+2,&r>a.
    \end{cases}
\]
So if $z = (a,y)$, the cycle notation of $z$ can be obtained from the cycle notation of $y$ by 
\begin{align*}
    z = (1,a+2)\iota_a(y).
\end{align*}

As a first application of this recurrence on fixed-point-free involutions, we obtain the following recurrences on the symplectic Rothe diagrams and match diagrams.

\begin{lem}
\label{lem:diagram-recursion}
Let $y \in \ifpf_{2m}$, $0 \leq a \leq 2m$, and $z = (a,y) \in \ifpf_{2m+2}$. Define $T_a:=\{(1,j):2\leq j\leq a+1\}$, then
\begin{equation}
\label{eq:MD-recursion}
    \MD(z)
    =
    T_a
    \cup
    \left\{
        (i+1,\iota_a(r)):
        (i,r)\in\MD(y)
    \right\},
\end{equation}
and
\begin{equation}
\label{eq:DSp-recursion}
    D^{\mathrm{Sp}}(z)
    =
    T_a
    \cup
    \left\{
        (\iota_a(i),\iota_a(j)):
        (i,j)\in D^{\mathrm{Sp}}(y)
    \right\}.
\end{equation}
\end{lem}

\begin{proof}
Let $y = (b_1, c_1) \cdots (b_{m}, c_{m})$, so $z = (1,a+2)(\iota_a(b_1), \iota_a(c_1)) \cdots(\iota_a(b_{m}), \iota_a(c_{m}))$.

For the match diagram, it is clear from definition that the first row of $\MD(z)$ is $T_a$. Since $\iota_a$ does not change relative order, for each cycle $(b_i,c_i)$ and letter $r \in [2m]$, $b_i<r<c_i$ if and only if $\iota_a(b_i)<\iota_a(r)<\iota_a(c_i)$. Thus row $i$ of $\MD(y)$ becomes row $i+1$ of $\MD(z)$, with every column relabeled by $\iota_a$. This proves \eqref{eq:MD-recursion}.

For the symplectic Rothe diagram, first observe that row $1$ consists
exactly of the cells $ (1,2),(1,3),\ldots,(1,a+1)$ as $z(1)=a+2$, and if $1<j<a+2$, then $z(j) \neq 1$, so $z(j)>1$. There are no cells in row or column $a+2$, since $z(a+2)=1$. Every other index has the form $\iota_a(r)$ for a unique $r\in[2m]$. For $r,s \in [2m]$, since $z(\iota_a(r))=\iota_a(y(r))$, we have
\[
\begin{aligned}
(\iota_a(r),\iota_a(s))\in D^{\mathrm{Sp}}(z)
&\Longleftrightarrow
\iota_a(r)<\iota_a(s)<z(\iota_a(r))
\text{ and }
\iota_a(r)<z(\iota_a(s))\\
&\Longleftrightarrow
r<s<y(r)
\text{ and }
r<y(s)\\
&\Longleftrightarrow
(r,s)\in D^{\mathrm{Sp}}(y).
\end{aligned}
\]
This proves \eqref{eq:DSp-recursion}.
\end{proof}

Match code also gives a direct characterization for when $z$ is $\mathrm{Sp}$-dominant.

\begin{prop}
\label{prop:Sp-dominant-match-code}
Let $z\in\ifpf_{2m}$. Then the following are equivalent:
\begin{itemize}
    \item[(1)] $z$ is $\mathrm{Sp}$-dominant;
    \item[(2)] there is a strict partition $\mu=(\mu_1>\mu_2>\cdots>\mu_k>0)$
    such that $\code(z)=(\mu_1,\ldots,\mu_k,0,\ldots,0)$;
    \item[(3)] there is a strict partition $\mu$ such that
    \[
        \MD(z)=D^{\mathrm{Sp}}(z)=\operatorname{Sh}(\mu),
    \]
    where $\operatorname{Sh}(\mu):= \{(i,i+r):1\leq i\leq k,\ 1\leq r\leq\mu_i\}$.
\end{itemize}
\end{prop}

\begin{proof}
We induct on $m$. The base case when $m=0$ is immediate as both $\MD(z)$ and $D^{\mathrm{Sp}}(z)$ are empty, corresponding to
the empty strict partition.

We first prove that $(2) \implies (3)$. Let $z = (a,y)$ and $\lambda := (\mu_2, \dots, \mu_k)$. By the inductive hypothesis, 
\begin{align*}
    \MD(y) = D^{\mathrm{Sp}}(y) = \operatorname{Sh}(\lambda).
\end{align*}

We claim that the relabeling map $\iota_a$ acts by adding $1$ to both
coordinates of every cell of $\operatorname{Sh}(\lambda)$. If $ (i,j)\in\operatorname{Sh}(\lambda)$, then $1\leq i\leq\ell(\lambda)$ and $i<j\leq i+\lambda_i$. Since $\lambda$ is strict, $ \lambda_i\leq\lambda_1-i+1$, so $j\leq i+\lambda_i\leq\lambda_1+1$. Then since $ a=\mu_1>\mu_2=\lambda_1$, we have $ i\leq\ell(\lambda)\leq\lambda_1<a$. It follows that $ \iota_a(i)=i+1$ and $\iota_a(j)=j+1$ for every cell $(i,j)\in\operatorname{Sh}(\lambda)$.

Define $ T_a:=\{(1,2),(1,3),\ldots,(1,a+1)\}$, by Lemma~\ref{lem:diagram-recursion}, 
\begin{align*}
    \MD(z) &= T_a \cup \{(i+1,\iota_a(j)):(i,j)\in\MD(y)\} \\
    &= T_a \cup \{(i+1,j+1):(i,j)\in\operatorname{Sh}(\lambda)\} = \operatorname{Sh}(\mu).\\
    D^{\mathrm{Sp}}(z)
    &= T_a \cup \{(\iota_a(i),\iota_a(j)): (i,j)\in D^{\mathrm{Sp}}(y)\}\\
    &= T_a \cup \{(i+1,j+1):(i,j)\in\operatorname{Sh}(\lambda)\} = \operatorname{Sh}(\mu).
\end{align*}

We next prove that $(1) \implies (2)$. Suppose that $z$ is $\mathrm{Sp}$-dominant, then there is a strict partition $\mu=(\mu_1>\mu_2>\cdots>\mu_k>0)$ such that $D^{\mathrm{Sp}}(z)=\operatorname{Sh}(\mu)$. By Lemma~\ref{l: two diagram bijection}, every nonempty row of $D^{\mathrm{Sp}}(z)$ is a row $b_i$, and row $i$ of $\MD(z)$ has exactly the same columns as row $b_i$ of $D^{\mathrm{Sp}}(z)$. Therefore, $b_i = i$ for $ 1\leq i \leq k$ and $\code(z) = (\mu_1, \mu_2, \dots, \mu_k, 0, \dots, 0)$.

Finally, $(3) \implies (1)$ follows immediately from Definition~\ref{d : sp-dominant}.
\end{proof}

\begin{rem}
    Another application of the match diagram is its connection to extended fpf-involution pipe dreams. Hamaker, Marberg, and Pawlowski~\cite{HMP22PD} introduced these objects to combinatorially compute symplectic Schubert and symplectic Grothendieck polynomials. Similar to how the left justified Rothe diagram is a pipe dream for the same permutation, the left justified (up to the diagonal) match diagram is also an involution pipe dream for the same fixed-point-free involution.
\end{rem}

\section{The maximal inverse Hecke atom}
\label{s: paired_words and tight insertion}
In this section, we construct the unique element of $\mathcal P(z)$ that maximizes
the Rajchgot index. First, we describe $\mathcal P(z)$ recursively
using paired word insertion. We then define tight insertion and
show that it uniquely minimizes the increase in $\mathcal J$.
Iterating this construction produces $\omega(\code(z))$ and proves
Theorem~\ref{t: main_top}.

\begin{defn}
\label{def: paired_word}
A \definition{paired word} of size $m$ is a permutation in $S_{2m}$ whose one-line notation is written as a sequence of increasing consecutive pairs
\[
    w=(b_1,c_1)(b_2,c_2)\cdots(b_m,c_m),
    \qquad b_i<c_i.
\]
The parentheses record consecutive blocks in the word and do not denote
cycles.

We say that a paired word $w$ is \definition{regular} if there are integers
\[
    h_1\geq h_2\geq\cdots\geq h_m = 2
\]
such that
\[
    L_w(b_i)=h_i,
    \qquad
    L_w(c_i)=h_i-1
    \qquad(1\leq i\leq m).
\]
The integer $h_i$ is the \definition{level} of the pair $(b_i,c_i)$ and the sequence $(h_1,\ldots,h_m)$ is called the \definition{level sequence} of $w$. It is immediate from the definition that $0 \leq h_i - h_{i+1} \leq 2$.
\end{defn}

For $z \in \ifpf_{2m}$, let $\Gamma(z)$ denote its canonical cycle notation, viewed as a paired word. We can then characterize of the set of inverse Hecke atoms using $\Gamma(z)$.

\begin{defn}
The relation \(\sim_{\mathrm p}\) of paired words is generated by the following replacements of two adjacent paired factors.
\begin{equation}
\label{eq:paired-local-relation}
    (A,D)(B,C)
    \ \sim_{\mathrm p}\ 
    (B,C)(A,D)
    \ \sim_{\mathrm p}\ 
    (B,D)(A,C),
    \quad \text{for }A<B<C<D.
\end{equation}
\end{defn}
\smallskip
\begin{prop}~\cite{MP20Ktheory}*{Proposition~3.15}
\label{p:paired-class-inverse-atoms}
For $z \in \ifpf_{2m}$, its set of inverse Hecke atoms is
\begin{align*}
    \mathcal P(z)
    =\{w: w\sim_{\mathrm p}\Gamma(z)\}.
\end{align*}
\end{prop}

Our goal now is to recursively construct $\PP(z)$ from $\PP(y)$ whenever $ z = (a,y)$ using the local moves in~\eqref{eq:paired-local-relation}.

\begin{defn}
\label{def:carrier-insertion}
Let $ w=(b_1,c_1)\cdots(b_m,c_m)$ be a paired word on $[2m]$. Let $0\leq a\leq 2m$ and $ q = a+2$. We now relabel $w$ by $\iota_a$ and prepend the pair $(1,q)$ to obtain the following new paired word of $[2m+2]$:
\begin{align*}
    (1,q)\iota_a(w) = (1,q)(\iota_a(b_1),\iota_a(c_1))\cdots
    (\iota_a(b_m),\iota_a(c_m)).
\end{align*}

We now generate a set of paired words from $(1,q)\iota_a(w)$ by the following relations. Let the pair containing $1$ be called the \definition{carrier} and its second entry the \definition{carrier partner}. Whenever the carrier $(1,D)$ is immediately followed by a pair $(B,C)$ with $ 1<B<C<D$, we may apply either of the \definition{carrier moves}
\begin{align*}
    \text{Swap}:\quad &(1,D)(B,C)\longmapsto(B,C)(1,D),\\
     \text{Cross} :\quad &(1,D)(B,C)\longmapsto(B,D)(1,C).
\end{align*}
After crossing, the carrier is now $(1,C)$.

Finally, let $\Ins_a(w)$ denote the set of all paired words obtained from $(1,q)\iota_a(w)$ by a finite sequence of legal carrier moves.
\end{defn}

For paired words that can be recursively constructed using $\Ins_a$, we may extract an insertion code.

\begin{defn}
\label{def:left-extraction}
Let $w$ be a paired word and consider the carrier $(1,R)$ and the pair immediately to its left $(B,C)$. Repeatedly apply the following extraction moves
\[
    \begin{cases}
        (B,C)(1,R) \longmapsto (1,R)(B,C) & \text{ if }1<B<C<R,\\
        (B,C)(1,R) \longmapsto (1,C)(B,R) & \text{ if }1<B<R<C
    \end{cases}
\]
until the carrier is the first pair. Notice that the extraction moves are exactly the inverses of the carrier moves. If at any step, $1 < R < B < C$, then we stop and declare $w$ \definition{inadmissible}. If the carrier $(1,R)$ reaches the first position, delete it, standardize the remaining paired word, record the integer $R-2$, and repeat the extraction recursively until reaching the empty word.

A paired word is \definition{admissible} if this process terminates at the empty word.  The recorded sequence is its \definition{extracted code}.
\end{defn}

\begin{lem}
\label{lem:critical-pair-invariance}
Let $w,u$ be paired words such that $ w \sim_p u$, then $w$ is admissible if and only if $u$ is admissible. Furthermore, if $w,u$ are admissible, then their extracted codes are equal.
\end{lem}

\begin{proof}
It is enough to consider $w,u$ that differ by one local relation \eqref{eq:paired-local-relation}.  If the relation involves the pair containing
$1$, then the three possible configurations are
\begin{align*}
    (1,D)(B,C),\qquad (B,C)(1,D),\qquad (B,D)(1,C), \qquad 1<B<C<D.
\end{align*}
The latter two extract in one step to the first, so all three have the same admissibility and 
extracted code.

A local relation~\eqref{eq:paired-local-relation} whose support is disjoint from the next extraction move commutes with that extraction move. Hence the only nontrivial overlap occurs when the local relation~\eqref{eq:paired-local-relation} involves the two pairs immediately to the left of the carrier. Let $1<A<B<C<D<E$. When the carrier partner is $E$, $D$, or $C$, the table below shows that any local relation~\eqref{eq:paired-local-relation} commutes with the extraction move. Each arrow records the result of moving the carrier left across the two displayed pairs.
\[
\resizebox{\textwidth}{!}{$
\renewcommand{\arraystretch}{1.35}
\begin{array}{c|c|c|c}
\text{CP}
&\text{first representative}
&\text{second representative}
&\text{third representative}\\ \hline
E&
(A,D)(B,C)(1,E)\mapsto(1,E)(A,D)(B,C)&
(B,C)(A,D)(1,E)\mapsto(1,E)(B,C)(A,D)&
(B,D)(A,C)(1,E)\mapsto(1,E)(B,D)(A,C)\\
D&
(A,E)(B,C)(1,D)\mapsto(1,E)(A,D)(B,C)&
(B,C)(A,E)(1,D)\mapsto(1,E)(B,C)(A,D)&
(B,E)(A,C)(1,D)\mapsto(1,E)(B,D)(A,C)\\
C&
(A,E)(B,D)(1,C)\mapsto(1,E)(A,D)(B,C)&
(B,D)(A,E)(1,C)\mapsto(1,E)(B,D)(A,C)&
(B,E)(A,D)(1,C)\mapsto(1,E)(B,D)(A,C)
\end{array}$}
\]
In every row the first extracted pair is $(1,E)$, and after deleting it the
three suffixes are equivalent under $\sim_p$. 

It remains to consider the two smaller possible partners. If the carrier partner is $A$, then extraction fails immediately. And if the carrier partner is $B$, then the first extraction only succeeds for the configuration $(C,D)(A,E)(1,B)$ and $(C,E)(A,D)(1,B)$. In either case, after the first extraction, the remaining word includes $(C,D)(A,B)$, which is inadmissible. 

These local checks show that any local relation~\eqref{eq:paired-local-relation} commutes with the extraction moves. So inadmissibility and the extracted code are also preserved by the local relations~\eqref{eq:paired-local-relation}.
\end{proof}

\begin{prop}
\label{p:Ins-factorization}
A paired word $w$ is admissible with extracted code $\code(z)$ if
and only if $ w \in \PP(z)$. Furthermore, if $z=(a,y)$, then
\begin{equation*}
    \PP(z)=\bigcup_{w\in\PP(y)}\Ins_a(w).
\end{equation*}
\end{prop}

\begin{proof}
Since $z=(a,y)$, $\Gamma(z)=(1,a+2)\,\iota_a(\Gamma(y))$. Then $\Gamma(z)$ is admissible with extracted code $\code(z)$ by induction. By Proposition~\ref{p:paired-class-inverse-atoms} and Lemma~\ref{lem:critical-pair-invariance}, every word in $\PP(z)$ is admissible with extracted code $\code(z)$.

Conversely, let $u$ be admissible with extracted code $\code(z)=(a,\code(y))$.  Its first extraction produces a paired word $w$ whose extracted code is $\code(y)$. By induction, $w\in\PP(y)$. Reversing the first extraction sequence is a sequence of carrier moves, so $u\in\Ins_a(w)$. Every carrier move is a local relation~\eqref{eq:paired-local-relation}, therefore
\begin{align*}
    u \sim_{p}(1,a+2)\iota_a(w) \sim_{p}(1,a+2)\iota_a(\Gamma(y)) =\Gamma(z).
\end{align*}
Hence $u\in\PP(z)$ by Proposition~\ref{p:paired-class-inverse-atoms}.
\end{proof}

We now show that $\Ins_a$ weakly increases the length of longest increasing subsequences starting at each letter.

\begin{lem}
\label{lem:insertion-lower-bounds}
Let $w$ be a paired word and $u \in \Ins_a(w)$. Then for every letter $r$ appearing in $w$,
\begin{equation}
\label{eq:old-letter-lower-bound}
    L_u(\iota_a(r))\geq L_w(r).
\end{equation}
Furthermore, let $q = a+2$, then the two new letters $1$ and $q$ satisfy 
\begin{equation}
\label{eq:new-letter-lower-bound}
    L_u(q)\geq\rho_w(a)+1 \, \text{ and }\, L_u(1)\geq\rho_w(a)+2.
\end{equation}
As a consequence,
\begin{equation}
\label{eq:J-universal-lower}
    \mathcal J(u)\geq \mathcal J(w)+2\rho_w(a)+3.
\end{equation}
\end{lem}

\begin{proof}
We prove by induction on carrier moves. The word $ u' = (1,q)\iota_a(w)$ satisfies~\eqref{eq:old-letter-lower-bound} as $\iota_a$ does not change relative order. Let $\ell = \rho_w(a)$ and let $v = (v_1, \dots, v_\ell)$ denote a longest increasing subsequence of $w|_{>a}$, then $\iota_a(v)$ is a longest increasing subsequence of $u'|_{>q}$. Then $u'$ also satisfies~\eqref{eq:new-letter-lower-bound} as $(1,q,\iota_a(v_1), \dots, \iota_a(v_\ell))$ is an increasing subsequence in $u'$. We now want to show that a single carrier move 
\begin{align*}
    &\text{Swap}:(1,D)(B,C)\mapsto(B,C)(1,D)\\
    \text{or }\quad &\text{Cross}:(1,D)(B,C)\longmapsto(B,D)(1,C),\quad \text{ for }1<B<C<D.
\end{align*}
preserves~\eqref{eq:old-letter-lower-bound} and~\eqref{eq:new-letter-lower-bound}. 

We first show that any increasing sequence not using $1$ or $q$ is preserved. If $D \neq q$, 
then changing the relative order from $D,B,C$ to $B,C,D$ or $B,D,C$ preserves every increasing subsequence using only old letters. If $D = q$, preservation of old subsequences is immediate, so~\eqref{eq:old-letter-lower-bound} is preserved.

The increasing sequence $(q,\iota_a(v_1), \dots, \iota_a(v_\ell))$ remains an increasing subsequence after any carrier moves since none of the moves can put $q$ after a larger number, so $L_u(q)\geq\rho_w(a)+1$. Let $R$ be the carrier partner. We claim that $(1,R,\iota_a(v_1), \dots, \iota_a(v_\ell))$ is always an increasing subsequence of $u$. We prove by induction. The base case when $R = q$ follows from above. The carrier partner changes every time we apply Cross, in which the new carrier partner $R'$ is always less than the old carrier partner $R$. So $(1,R',\iota_a(v_1), \dots, \iota_a(v_\ell))$ is again an increasing subsequence of $u$. Hence~\eqref{eq:new-letter-lower-bound} is preserved.

Finally,~\eqref{eq:J-universal-lower} follows from summing over~\eqref{eq:old-letter-lower-bound} and~\eqref{eq:new-letter-lower-bound} across all letters.
\end{proof}

When $w$ is regular, we can pick out a specific element in $\Ins_a(w)$ such that all the inequalities in Lemma~\ref{lem:insertion-lower-bounds} are equalities.

\begin{defn}
\label{d: tight insertion}
Let $w$ be a regular paired word with level sequence $h_1\geq\cdots\geq h_m$ and let $H=\rho_w(a)+2$. Define the \definition{tight insertion $\TIns_a(w)$} as the element in $\Ins_a(w)$ obtained by scanning the old pairs from left to right and applying
\[
\begin{cases}
   \text{Swap}:(1,D)(B,C)\longmapsto(B,C)(1,D) & \text{if } H < h_i, \\
   \text{Cross}:(1,D)(B,C)\longmapsto(B,D)(1,C) & \text{if } H = h_i,\\
   \text{Stop}& \text{if } H > h_i.
\end{cases}
\]
\end{defn}

A priori, it is unclear that $\TIns_a$ produces legal carrier moves at every step, so we first show that $\TIns_a$ is well defined.

\begin{lem}
    \label{l: TIns-well-define}
    Let $w=(b_1,c_1)\cdots(b_m,c_m)$ be regular with level sequence $h_1\geq\cdots\geq h_m$. Then $ u = \TIns_a(w)$ is well-defined and no further carrier move is legal after the procedure terminates.
\end{lem}
\begin{proof}
 Let $R$ denote the carrier partner. In order for the carrier moves to be legal, we want to show that the carrier $(1,R)$ nests the pair to its right $(\iota_a(b_i), \iota_a(c_i))$ whenever $h_i \geq H$.

When $h_i > H$, the carrier partner is $q = a+2$. If $\iota_a(c_i)>q$, then $c_i > a$, so every increasing subsequence in $w$ beginning with $c_i$ has length at most $\rho_w(a) = H-2$. This is a contradiction to 
\[
    L_w(c_i)=h_i-1\geq H-1.
\]
Hence $\iota_a(c_i)<q$, so $(1,q)$ nests the pair $(\iota_a(b_i), \iota_a(c_i))$.

When $h_i = H$, the carrier partner is either $q$ or $\iota_a(c_{i-1})$. If $R = q$, then by the same argument above, $(1,q)$ nests the pair $(\iota_a(b_i), \iota_a(c_i))$. Now if $R = \iota_a(c_{i-1})$, then $h_{i-1} = H$. Suppose that $\iota_a(c_{i-1}) < \iota_a(c_{i})$, then 
\[
       L_w(c_{i-1})\geq 1+L_w(c_i) = h_i = H,
\]
which contradicts $L_w(c_{i-1}) = h_{i-1}-1 = H-1$. So $\iota_a(c_{i-1}) > \iota_a(c_{i})$, and $(1, R)$ nests the pair $(\iota_a(b_i), \iota_a(c_i))$.

Finally, we want to show that no additional legal moves can be made. Let $t = \max(\{i: h_i \geq H\} \cup \{0\})$. If $t = m$, there is nothing to prove. Suppose $ t<m$, and let \((1,R)\) be the final carrier. Then either $R=q=a+2$ or $R=\iota_a(c_t)$. In the second case, the carrier partner changing implies that $q > \iota_a(c_t)$. Thus in both cases, $R \leq q$. Assume for contradiction that the carrier nests the next pair, i.e. $\iota_a(c_{t+1}) < R$ and $c_{t+1}\leq a$. Choose a letter \(r>a\) such that $L_w(r)=\rho_w(a)=H-2$. If \(r\) is the first entry of its pair, that pair has level \(H-2\); if \(r\) is the second entry, its pair has level \(H-1\). In either case its pair occurs strictly after pair \(t\). It cannot occur in pair \(t+1\), since both entries of that pair are at most \(a\). Hence \(r\) occurs strictly after \(c_{t+1}\). Therefore $b_{t+1}<c_{t+1}\leq a<r$ and adjoining \(b_{t+1},c_{t+1}\) to a longest increasing subsequence beginning at \(r\) gives an increasing subsequence of length $2+(H-2)=H$ starting at \(b_{t+1}\). This contradicts $L_w(b_{t+1})=h_{t+1}<H$. Thus the final carrier cannot nest the next pair and no more carrier moves can be made.
\end{proof}

\begin{lem}
\label{lem:tight-insertion-exact}
Let $w$ be  a regular paired word with level sequence $h_1\geq\cdots\geq h_m$. Let $ H=\rho_w(a)+2$, $q = a+2$, and $u=\TIns_a(w)$, then the following hold
\begin{itemize}
    \item $L_u(\iota_a(r))=L_w(r)$ for all letters $r$ in $w$,
    \item $L_u(1)=H$,
    \item $L_u(q)=H-1$, and
    \item $\mathcal J(u)=\mathcal J(w)+2\rho_w(a)+3.$
\end{itemize}
Furthermore, $u$ is regular, and its level sequence is obtained by inserting
one additional part $H$ into the weakly decreasing level sequence of $w$.
\end{lem}

\begin{proof} 
Let $w = (b_1, c_1) \cdots (b_m,c_m)$ and let
\begin{align*}
    s=\max\{i:h_i>H\}, \qquad t=\max\{i:h_i\geq H\},
\end{align*}
with either maximum interpreted as $0$ when the corresponding set is empty. When \(s<t\), applying the tight insertion yields
\begin{align*}
     u = &(\iota_a(b_1),\iota_a(c_1)) \cdots (\iota_a(b_s),\iota_a(c_s))\\
    &(\iota_a(b_{s+1}),q)(\iota_a(b_{s+2}),\iota_a(c_{s+1}))\cdots (\iota_a(b_t),\iota_a(c_{t-1}))(1,\iota_a(c_t))\\
    &(\iota_a(b_{t+1}), \iota_a(c_{t+1})) \cdots (\iota_a(b_m),\iota_a(c_m)).
\end{align*}
 When \(s=t\), replace the middle block by \((1,q)\). We refer to the first $s$ pairs as the first block, the next $t-s+1$ pairs as the second block, and the remaining $m-t$ pairs as the third block, corresponding to the rows in the presentation above.

Then the carrier moves of $\TIns_a(w)$ can be characterized as follows: 
In the first block, $h_i > H$, so we applied Swap. In the second block, $h_i = H$, so we applied Cross. In the third block, $h_i < H$, we applied nothing.

\textbf{Third block.} Pairs in the third block have exactly the same suffix as before, so their starting-LIS lengths are unchanged. In addition, for any letter $\iota_a(r)$ in the third block, $L_u(\iota_a(r)) = L_w(r) < H$.

\textbf{Second block.} In the second block, if $t>s$, then the second entries satisfy
\[
    q>\iota_a(c_{s+1})>\iota_a(c_{s+2})>\cdots>\iota_a(c_t).
\]
Since $q$ is the largest number in the second block, its increasing subsequence can only use letters greater than $q$ which are all in the third block, so $L_u(q)= \rho_w(a) + 1 = H-1$. For each $\iota_a(c_i)$ in the second block, every letter to its right in the second block is smaller than it, so $L_u(\iota_a(c_i))\leq L_w(c_i)=H-1$.
By Lemma~\ref{lem:insertion-lower-bounds}, $L_u(\iota_a(c_i)) \geq L_w(c_i) = h_i - 1 =  H-1$, so equality follows.

For $s <  i<j\leq t$, we have $b_i>b_j$. Otherwise $b_i<b_j$, and since $b_j$ occurs later with $L_w(b_j) = H$, an increasing subsequence beginning $b_i,b_j$ would have length at least $H+1$, contradicting $L_w(b_i)=H$. So the longest increasing subsequence in $u$ starting with $\iota_a(b_i)$ cannot use any other $b_j$ in the second block. Therefore, the second entry in this LIS has starting-LIS of length at most $H-1$, so $L_u(\iota_a(b_i)) \leq H$. Then by Lemma~\ref{lem:insertion-lower-bounds}, $L_u(\iota_a(b_i))\geq L_w(b_i) = h_i = H$, so equality follows.

The final letter $1$ in the second block has $L_u(1) = H$ as the carrier partner is $\iota_a(c_t)$ with $L_u(\iota_a(c_t)) = H-1$. $L_u(1)$ cannot be any larger since every letter $r$ to the right of $1$ has $L_u(r) \leq H-1$.

If $s = t$, then the second block is simply $(1,q)$. By the same arguments as above, $L_u(q) = H-1$ and $L_u(1) \geq 1 + L_u(q) = H$. $L_u(1)$ cannot be any larger since every letter $r$ to the right of $1$ has $L_u(r) \leq H-1$.

\textbf{First block.} It remains to check the LIS starting at each letter in the first block. We first induct on $\iota_a(c_1), \dots, \iota_a(c_s)$ from right to left. For $1 \leq i < j \leq s$, if $h_i = h_j$, then $c_i > c_j$. Suppose not, then $c_j$ is both larger than $c_i$ and to the right of $c_i$, so $h_i - 1 =  L_w(c_i) > L_w(c_j) = h_j-1$, which is a contradiction. Therefore any longest increasing subsequence in $u$ starting with $\iota_a(c_i)$ must have second letter $\iota_a(r)$ such that $L_u(\iota_a(r)) = L_w(r) <  h_i - 1$, so $L_u(\iota_a(c_i)) = L_u(\iota_a(r)) + 1 = L_w(r) + 1 \leq h_i - 1$.  Then by Lemma~\ref{lem:insertion-lower-bounds}, $L_u(\iota_a(c_i))\geq L_w(c_i) = h_i -1 \geq H$, so equality follows. 

We now induct on $\iota_a(b_1), \dots, \iota_a(b_s)$ from right to left. By a similar argument, for $1 \leq i < j \leq s$, if $h_i = h_j$, then $b_i > b_j$. So any longest increasing subsequence in $u$ starting with $\iota_a(b_i)$ must have second letter $\iota_a(r)$ with $L_u(\iota_a(r)) =  L_w(r) <  h_i$, so $L_u(\iota_a(b_i)) = L_u(\iota_a(r)) + 1 = L_w(r) + 1 \leq h_i$. Then by Lemma~\ref{lem:insertion-lower-bounds}, $L_u(\iota_a(b_i))\geq L_w(b_i) = h_i$, so equality follows.

Finally, since $L_u(1)=H$, $L_u(q)=H-1$, and $L_u(\iota_a(r))=L_w(r)$ for all letters $r$ in $w$, summing over all letters in $u$ gives 
\begin{align*}
    \mathcal J(u)=\mathcal J(w)+2\rho_w(a)+3.
\end{align*}
Furthermore, since $w$ is regular, one can check that $u$ is also regular with level sequence
\[
    h_1,\ldots,h_s,
    \underbrace{H,\ldots,H}_{t-s+1\text{ copies}},
    h_{t+1},\ldots,h_m,
\]
which is weakly decreasing.
\end{proof}

\begin{exa}
Let $w=(1,2)(3,4)(5,6)(7,8)$, its level sequence is $(8,6,4,2)$.  Let $a=4$, then $\rho_w(4)=4$ and $H=6$. Relabeling gives $(1, a+2)\iota_a(w) = (1,6)(2,3)(4,5)(7,8)(9,10)$. Then $\TIns_a(w)$ is obtained by applying the following sequence of carrier moves
\[
\begin{aligned}
    &(1,6)(2,3)(4,5)(7,8)(9,10)\\
    &\quad\longmapsto
    (2,3)(1,6)(4,5)(7,8)(9,10)\\
    &\quad\longmapsto
    (2,3)(4,6)(1,5)(7,8)(9,10).
\end{aligned}
\]
Thus $ \TIns_4(w)=(2,3)(4,6)(1,5)(7,8)(9,10)$. It is regular with level sequence $(8,6,6,4,2)$.
\end{exa}

The tight-insertion $\TIns_a(w)$ gives the unique element in $\Ins_a(w)$ that minimizes $\mJ$.

\begin{prop}
\label{p:strict-tight-insertion}
Let $w$ be regular and let $u\in\Ins_a(w)$.  Then
\[
    \mathcal J(u)\geq\mathcal J(w)+2\rho_w(a)+3,
\]
with equality if and only if $ u=\TIns_a(w)$.
\end{prop}

\begin{proof}
The inequality is Lemma~\ref{lem:insertion-lower-bounds}, and
Lemma~\ref{lem:tight-insertion-exact} shows that $\TIns_a(w)$ achieves
equality. We prove that no other element in $\Ins_a(w)$ does.

Let $H=\rho_w(a)+2$ and $q =a+ 2$. Suppose equality holds, then 
\begin{align*}
     0& = \mathcal J(u) - \mathcal J(w) - 2\rho_w(a) - 3 \\
     &= \sum_{r }
\bigl(L_u(\iota_a(r))-L_w(r)\bigr) +\bigl(L_u(q)-(H-1)\bigr)
 +\bigl(L_u(1)-H\bigr).
\end{align*}
By Lemma~\ref{lem:insertion-lower-bounds}, each term on the right hand side is nonnegative, so
\begin{equation}
\label{eq:equality-old-L}
    L_u(\iota_a(r))=L_w(r)
    \qquad\text{for every letter }r \text{ in }w,
\end{equation}
and
\begin{equation}
\label{eq:equality-new-L}
    L_u(q)=H-1,
    \qquad
    L_u(1)=H.
\end{equation}

Fix a sequence of carrier-moves that produce \(u\), and let \(i\) be the first old pair at which its choice differs from the tight-insertion rule.

\smallskip
\noindent\emph{Case 1: $h_i>H$.}
Tight insertion swaps past pair $i$.  Since the level sequence is weakly
decreasing, every earlier pair also has level greater than $H$, so no
crossing has yet occurred and the carrier is $(1,q)$. 

If the insertion stops, then $1$ precedes the larger letter $\iota_a(b_i)$, but
\[
    L_u(1)\geq1+L_u(\iota_a(b_i))=h_i+1>H,
\]
contradicts~\eqref{eq:equality-new-L}.

Suppose instead that the carrier crossed pair $i$. The new carrier partner
is $\iota_a(c_i)$, with $L_u(\iota_a(c_i)) = h_i-1$ by~\eqref{eq:equality-old-L}. Let $h=h_i$ be the current level. As long as the carrier only swaps, the carrier partner is unchanged. If it crosses another pair of level $h$, the new carrier partner still satisfies $L_u(\iota_a(c_j)) = h-1$. Thus, if the carrier never crosses with a pair $(\iota_a(b_j), \iota_a(c_j))$ such that $h_j < h$, then the carrier partner $R$ gives $L_u(1)\geq1+L_u(R)=h>H$, which is a contradiction. Therefore, such crossing with a lower level pair must occur. Let $(b_j,c_j)$ be the first pair of level $h_j<h$ through which the carrier crosses.  Immediately before this crossing, the carrier $R$ satisfies $L_u(R)=h-1$ and $\iota_a(b_j)<R$. The crossing puts $R$ immediately after $\iota_a(b_j)$, so
\[
     h > h_j = L_w(b_j)=L_u(\iota_a(b_j))\geq1+L_u(R)=h,
\]
which is a contradiction.

\smallskip
\noindent\emph{Case 2: $h_i=H$.}
Tight insertion crosses pair $i$. Let $(1,R)$ be the carrier immediately
before that pair.  Because every earlier choice was tight, either $R=q$ or
$R$ is the second entry of an earlier level-$H$ pair.  In either case, $L_u(R)= h_i - 1  =H-1$.
If the carrier stops instead of crossing, then
\[
    L_u(1)\geq1+L_u(\iota_a(b_i)) = H+1,
\]
is a contradiction. If the carrier swaps instead, then $\iota_a(c_i)<R$, and $R$ remains to the right of $\iota_a(c_i)$. Hence,
\[
    L_u(\iota_a(c_i))\geq1+L_u(R)=H,
\]
which is a contradiction to $L_u(\iota_a(c_i))=L_w(c_i)=H-1$.

\smallskip
\noindent\emph{Case 3: $h_i<H$.}
By Lemma~\ref{l: TIns-well-define}, after all preceding tight choices the
carrier does not nest pair $i$.  Therefore neither carrier move is legal,
and stopping is forced.

Thus equality forces every carrier move to follow the tight-insertion rule, and hence \(u=\TIns_a(w)\). 
\end{proof}

Using the tight-insertion, we can recursively construct the unique inverse Hecke atom that maximizes the Rajchgot index.
\begin{defn}
\label{def:omega-recursion}
For a valid code $c$, define a paired word $\omega(c)$ recursively by
\[
    \omega(\varnothing)=\varnothing,
    \qquad
    \omega(a,c')=\TIns_a(\omega(c')).
\]
Thus the entries of $c=(c_1,\ldots,c_m)$ are read from right to left.
\end{defn}

\begin{exa}
For $c=(3,0,2,0)$, the recursion gives
\begin{align*}
    \omega(0)&=(1,2),\\
    \omega(2,0)&=(2,4)(1,3),\\
    \omega(0,2,0)&=(1,2)(4,6)(3,5),\\
    \omega(3,0,2,0)&=(2,5)(1,3)(6,8)(4,7).
\end{align*}
\end{exa}

To show that $\omega(c)$ is well-defined, it suffices to show the following lemma.
\begin{lem}
\label{lem:omega-regular}
For a valid code $c$, the paired word $\omega(c)$ is regular.
\end{lem}

\begin{proof}
We induct on the length of $c$. The empty paired word is regular.  Write
$c=(a,c')$.  By induction, $\omega(c')$ is regular.  By Lemma
\ref{lem:tight-insertion-exact}, $\TIns_a(\omega(c')) = \omega(c)$ is again regular.
\end{proof}

\begin{lem}
\label{lem:omega-inverse-atom}
For $z \in \ifpf_{2m}$, $\omega(\code(z))\in\mathcal P(z)$.
\end{lem}

\begin{proof}
We induct on the number of cycles of $z$. The base case when $z = (1,2)$ is trivial. Write $z=(a,y)$.  By induction, $\omega(\code(y))\in\mathcal P(y)$.
Then by Proposition~\ref{p:Ins-factorization},
\[
    \omega(\code(z))
    =\TIns_a(\omega(\code(y)))
    \in\Ins_a(\omega(\code(y))) \subseteq \PP(z). \qedhere
\]
\end{proof}

Finally, we show that $\omega$ recursively constructs the unique inverse Hecke atom that maximizes the Rajchgot index.

\begin{lem}
\label{lem:omega-coordinatewise-minimal}
Let $z \in \ifpf_{2m}$, $u = \omega(\code(z))$, and $w\in\mathcal P(z)$ be any inverse Hecke atom. Then
\begin{equation}
\label{eq:omega-coordinatewise-minimal}
    L_{u}(r)\leq L_w(r) \qquad\text{for  }r\in[2m].
\end{equation}
\end{lem}

\begin{proof}
We prove by induction on the number of cycles. The base case when $z$ is empty is trivial.  Write $z=(a,y)$ and set $ u'=\omega(\code(y))$.
By Proposition~\ref{p:Ins-factorization}, there exists some $ w' \in\mathcal P(y)$ such that
$w\in\Ins_a(w')$.  The induction hypothesis gives
\begin{equation}
\label{eq:parent-coordinatewise-minimal}
    L_{u'}(r)\leq L_{w'}(r)
    \qquad \text{for }r\in[2m-2].
\end{equation}
By Lemma~\ref{lem:insertion-lower-bounds} and Lemma~\ref{lem:tight-insertion-exact}, for every letter $r$ appearing in $u'$,
\begin{align*}
    L_{u}(\iota_a(r))=L_{u'}(r) \leq L_{w'}(r) \leq L_w(\iota_a(r)).
\end{align*}

It remains to compare the two letters $1$ and $q = a+2$. Since $L_{u'}(r)\leq L_{w'}(r)$ for all $r$, we have $\rho_{u'}(a)\leq\rho_{w'}(a)$. Therefore by Lemma~\ref{lem:insertion-lower-bounds} and Lemma~\ref{lem:tight-insertion-exact},
\begin{align*}
    L_{u}(q) &=\rho_{u'}(a)+1 \leq\rho_{w'}(a)+1 \leq L_w(q),\\
    L_{u}(1) &=\rho_{u'}(a)+2 \leq\rho_{w'}(a)+2 \leq L_w(1). \qedhere
\end{align*}
\end{proof}

\begin{prop}
\label{p:unique-raj-maximizer}
Let $z \in \ifpf_{2m}$, then $\omega(\code(z))$ is the unique minimizer of $\mathcal J$ in $\PP(z)$. Equivalently, it is the unique maximizer of the Rajchgot index among all permutations in $\PP(z)$.
\end{prop}

\begin{proof}
We induct on the number of cycles.  The empty case is immediate.  Write
$z=(a,y)$ and let $w'=\omega(\code(y))$ and $w=\omega(\code(z))=\TIns_a(w')$.
Now take arbitrary $u\in\PP(z)$. By Proposition~\ref{p:Ins-factorization}, there exists $u'\in\PP(y)$ such that $u\in\Ins_a(u')$. For each pair of $u$ and $u'$ satisfying $u\in\Ins_a(u')$, we define
\begin{align*}
    \epsilon(u,u',a) :=\mathcal J(u)-\mathcal J(u')-2\rho_{u'}(a) - 3 .
\end{align*}
By Lemma~\ref{lem:tight-insertion-exact},
\[
    \mathcal J(w)=\mathcal J(w')+2\rho_{w'}(a)+3.
\]
Therefore
\begin{equation}
    \label{eq:uniqueness-decomposition}
    \mathcal J(u) - \mathcal J(w)
    = \,\epsilon(u,u',a) +\bigl(\mathcal J(u')-\mathcal J(w')\bigr) +2\bigl(\rho_{u'}(a)-\rho_{w'}(a)\bigr).
\end{equation}

By Lemma~\ref{lem:insertion-lower-bounds}, $\epsilon(u,u',a)$ is nonnegative. By the inductive hypothesis, $\mathcal J(u')-\mathcal J(w')$ is nonnegative. By Lemma~\ref{lem:omega-coordinatewise-minimal}, $\rho_{u'}(a)-\rho_{w'}(a)$ is also nonnegative. Thus $\mathcal J(u) \geq \mathcal J(w)$.

Suppose equality holds, then all three nonnegative terms in
\eqref{eq:uniqueness-decomposition} vanish.  In particular, by the inductive hypothesis, $\mathcal J(u')=\mathcal J(w')$ implies that $ u' = w'$. Furthermore, $\epsilon(u,u',a)=0$ implies that $u = \TIns_a(u')$ by Proposition~\ref{p:strict-tight-insertion}. So $u = \TIns_a(u') = \TIns_a(w') = w$ and $w = \omega(\code(z))$ is the unique $\mathcal J$-minimizer.

Finally, for a word $u \in \PP(z)$,
\[
    \raj(u) = \binom{2m+1}{2}-\mathcal J(u).
\]
So $w = \omega(\code(z))$ is also the unique maximizer of the Rajchgot index. 
\end{proof}

We can now prove Theorem~\ref{t: main_top}.
\begin{proof}[Proof of Theorem~\ref{t: main_top}]
    By Theorem~\ref{t: psw raj} and Proposition~\ref{p:unique-raj-maximizer}, $\omega(\code(z))^{-1}$ is the unique element in $B_{\mathrm{FPF}}(z)$ that maximizes the Rajchgot index, so its Grothendieck polynomial has the highest degree among all elements in $B_{\mathrm{FPF}}(z)$. The statement then follows from Theorem~\ref{t: symplectic to Gro expansion}.
\end{proof}

Recall that the highest-degree homogeneous component of a Grothendieck polynomial is completely determined, up to scalar multiplication, by its Rajchgot code. In fact, Pechenik, Speyer, and Weigandt showed that among the $n!$ Grothendieck polynomials indexed by permutations in $S_n$, the number of distinct highest-degree homogeneous components, up to scalar multiplication, is given by the number of set partitions of $[n]$, which is the $n\textsuperscript{th}$ Bell number. In contrast, the highest-degree homogeneous components of symplectic Grothendieck polynomials exhibit the opposite behavior: they are pairwise distinct even up to scalar multiplication.

\begin{thm}
\label{t: fpf distinct top}
    For $z,z' \in \ifpf_{2m}$, 
    \begin{align*}
        \widehat{\fG}_{z}^\fpf(\x) = c \cdot \widehat{\fG}_{z'}^\fpf(\x) \, \text{ for some } c \in \mathbb{Q} \quad \iff \quad z = z'
    \end{align*}
\end{thm}

By Theorem~\ref{t: main_top} and Theorem~\ref{t: rajcode distinguish}, it suffices to show that the map from $\ifpf_{2m}$ to weak compositions that sends $z \mapsto \rajcode(\omega(\code(z))^{-1})$ is injective. We show injectivity by constructing the inverse map. The first step is to recover the first entry of $\code(z)$.

\begin{lem}
\label{l: recover first entry}
    For $z \in \ifpf_{2m}$, let $u = \omega(\code(z))$ and $\rajcode(u^{-1}) = (r_1, \dots, r_{2m})$. Then the first entry of $\code(z)$ is 
    \begin{align*}
        \code_1(z) = \max\{i \in [2m]: i + r_i = r_1 + 2\} - 2.
    \end{align*}
    In particular, it is determined by $\rajcode(u^{-1})$.
\end{lem}
\begin{proof}
    Let $z = (a,y)$, $ w = \omega(\code(y))$, $u = \TIns_a(w) = \omega(\code(z))$, $q = a+2$, and $H= L_u(1) =\rho_w(a) +2$. By Lemma~\ref{lem:tight-insertion-exact}, $L_u(q) = H-1$, and for any $\iota_a(s) > q$, $L_u(\iota_a(s)) = L_w(s) \leq \rho_w(a) = H-2$. So $q$ is the largest letter whose starting LIS has length $H-1$, i.e. $q = \max\{i \in [2m]:L_u(i) = H-1 = L_u(1)-1\}$. Recall that $L_u(i) = 2m - i + 1 - r_i$, so 
    \begin{align*}
        L_u(i) = L_u(1) -1 \iff 2m - i + 1 - r_i = 2m - r_1 -1 \iff i + r_i = r_1 +2.
    \end{align*}
    Hence, the first entry of $\code(z)$ is 
    \begin{align*}
        &\code_1(z) =  a = q-2 = \max\{i \in [2m]: i + r_i = r_1 + 2\} - 2. \qedhere
    \end{align*}
\end{proof}

\begin{proof}[Proof of Theorem~\ref{t: fpf distinct top}]
We show that the map $z \mapsto \rajcode(\omega(\code(z))^{-1})$ is injective by constructing the inverse map. We induct on $m$, the base case when $m=0$ is trivial. Let $z = (a,y)$, $w = \omega(\code(y))$, $u = \TIns_a(w) = \omega(\code(z))$. Let $\rajcode(u^{-1}) = (r_1, \dots, r_{2m})$ and $\rajcode(w^{-1}) = (r_1', \dots, r_{2m-2}')$ be the Rajchgot codes. By Lemma~\ref{l: recover first entry}, $\rajcode(u^{-1})$ completely determines the first entry $a$. It remains to show that $\rajcode(u^{-1})$ determines $\rajcode(w^{-1})$ as well. By Lemma~\ref{lem:tight-insertion-exact}, for $s \in [2m-2]$,
\begin{align*}
    r_s' &= 2m - s -1 - L_w(s) =  2m - s -1 - L_u(\iota_a(s))\\
    &= 2m - s -1 - (2m - \iota_a(s) + 1 - r_{\iota_a(s)})\\
    & = r_{\iota_a(s)} + \iota_a(s) - s - 2.
\end{align*}
So $\rajcode(u^{-1})$ determines $\rajcode(w^{-1})$, which determines $\code(y)$ by the inductive hypothesis. Finally, $\code(z) = (a,\code(y))$, so $\rajcode(u^{-1})$ determines $\code(z)$.
\end{proof}

\section{Symplectic Rajchgot index}
\label{sec:snow-degree}

The goal of this section is to define the symplectic Rajchgot index using the match diagram and to
prove Theorem~\ref{t: main_degree} and Corollary~\ref{cor: Cm-reg}.
\begin{defn}
\label{def:snow-sraj}
Let $z=(b_1,c_1)\cdots(b_m,c_m)\in\ifpf_{2m}$. We construct the \definition{clouds} of \(\MD(z)\) by processing its rows from bottom to top. A cell \((i,j)\in\MD(z)\) is called
\definition{free} if no cloud has previously been placed in column \(j\).
If row \(i\) contains a free cell, then we mark its rightmost free cell as
a \definition{dark cloud}.  If row \(i\) contains at least two free
cells, then we also place an \definition{invisible dark cloud} at
\((i,b_i)\).

The \definition{snow diagram} of \(z\), denoted by \(\snow(z)\), is
obtained from \(\MD(z)\) by filling every cell weakly above a cloud cell in the same column.  Thus
\[
    \snow(z) = \MD(z) \cup \left\{ (r,j): \text{there is a cloud at }(i,j)\text{ with }1\le r\le i \right\}.
\]

Finally, the \definition{symplectic Rajchgot index} of \(z\) is the number of cells that are not invisible dark clouds in $\snow(z)$.
\end{defn}

\begin{exa}
\label{exa:snow-running}
Let $ z=(1,5)(2,3)(4,8)(6,7)$. We present $\MD(z)$ with its lasers along with the snow diagram $\snow(z)$.
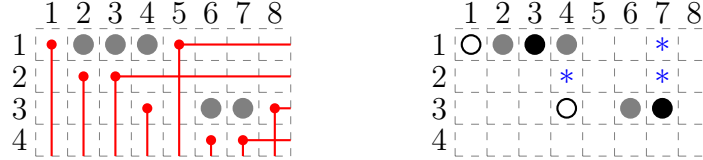
\[
\begin{tikzpicture}
            [x=1em,y=1em,thick,color = blue]
            \draw[step=1,gray,ultra thin,dashed] (0,0) grid (8,4);
            \node[color=black] at (-0.5,3.5) {$1$};
            \node[color=black] at (-0.5,2.5) {$2$};
            \node[color=black] at (-0.5,1.5) {$3$};
            \node[color=black] at (-0.5,0.5) {$4$};
            \node[color=black] at (0.5,4.5) {$1$};
            \node[color=black] at (1.5,4.5) {$2$};
            \node[color=black] at (2.5,4.5) {$3$};
            \node[color=black] at (3.5,4.5) {$4$};
            \node[color=black] at (4.5,4.5) {$5$};
            \node[color=black] at (5.5,4.5) {$6$};
            \node[color=black] at (6.5,4.5) {$7$};
            \node[color=black] at (7.5,4.5) {$8$};
            \filldraw [red] (0.5,3.5) circle (1.5pt);
            \filldraw [red] (4.5,3.5) circle (1.5pt);
            \filldraw [red] (1.5,2.5) circle (1.5pt);
            \filldraw [red] (2.5,2.5) circle (1.5pt);
            \filldraw [red] (3.5,1.5) circle (1.5pt);
            \filldraw [red] (7.5,1.5) circle (1.5pt);
            \filldraw [red] (5.5,0.5) circle (1.5pt);
            \filldraw [red] (6.5,0.5) circle (1.5pt);
            \draw[color=red] (0.5,3.5)--(0.5, 0);
            \draw[color=red] (4.5,3.5)--(4.5, 0);
            \draw[color=red] (4.5,3.5)--(8, 3.5);
            \draw[color=red] (1.5,2.5)--(1.5, 0);
            \draw[color=red] (2.5,2.5)--(2.5, 0);
            \draw[color=red] (2.5,2.5)--(8, 2.5);
            \draw[color=red] (3.5,1.5)--(3.5, 0);
            \draw[color=red] (7.5,1.5)--(7.5, 0);
            \draw[color=red] (7.5,1.5)--(8, 1.5);
            \draw[color=red] (5.5,0.5)--(5.5,0);
            \draw[color=red] (6.5,0.5)--(6.5,0);
            \draw[color=red] (6.5,0.5)--(8,0.5);
            \filldraw [gray] (1.5,3.5) circle (3.5pt);
            \filldraw [gray] (2.5,3.5) circle (3.5pt);
            \filldraw [gray] (3.5,3.5) circle (3.5pt);
            \filldraw [gray] (5.5,1.5) circle (3.5pt);
            \filldraw [gray] (6.5,1.5) circle (3.5pt);
        \end{tikzpicture}\qquad \qquad 
\begin{tikzpicture}[x=1em,y=1em,thick]
    \draw[step=1,gray,ultra thin,dashed] (0,0) grid (8,4);
    \foreach \r/\yy in {1/3.5,2/2.5,3/1.5,4/0.5}
        \node at (-0.5,\yy) {$\r$};
    \foreach \c in {1,...,8}
        \node at (\c-0.5,4.5) {$\c$};
    \foreach \x/\y in {2/3,4/3,6/1}
        \filldraw[gray] (\x-0.5,\y+0.5) circle (3.5pt);
    \foreach \x/\y in {3/3,7/1}
        \filldraw[black] (\x-0.5,\y+0.5) circle (3.5pt);
    \draw[black] (0.5,3.5) circle (3.5pt);
    \draw[black] (3.5,1.5) circle (3.5pt);
    \foreach \x/\y in {4/2,7/2,7/3}
        \node[blue] at (\x-0.5,\y+0.5) {$\asterisk$};
\end{tikzpicture}
\]
The visible dark clouds are at \((3,7)\) and \((1,3)\), while the
invisible dark clouds are at \((3,4)\) and \((1,1)\). Every empty cell above each of the dark clouds is filled with the snow cell $\color{blue}{\asterisk}$. The snow diagram $\snow(z)$ has eight cells that are not invisible dark clouds, so $\sraj(z)=8$.
\end{exa}

Let $\mathcal C(z)\subseteq[2m]$ denote the set of columns containing a visible or invisible dark cloud. We now present a recursion on this set. For a subset $C\subseteq[2m]$ and $0 \leq a \leq 2m$, define $F_a(C):=[a]\setminus C$ and $\widehat{F}_a(C) := \max F_a(C)$ if $F_a(C) \neq \varnothing$.
Define
\begin{equation*}
\begin{aligned}
    \Phi_a(C)
    :={}&\iota_a(C)  \, \cup \,
    \begin{cases}
        \{\iota_a(\widehat{F}_a(C))\},&F_a(C)\neq\varnothing,\\
        \varnothing,&F_a(C)=\varnothing,
    \end{cases}  \, \cup \, 
    \begin{cases}
        \{1\},&|F_a(C)|\ge2,\\
        \varnothing,&|F_a(C)|\le1.
    \end{cases}
\end{aligned}
\end{equation*}

\begin{lem}
\label{lem:cloud-recurrence}
Let \(y\in\ifpf_{2m}\), \(0\le a\le2m\), and $ z=(a,y)\in\ifpf_{2m+2}$.
Then $\mathcal C(z)=\Phi_a(\mathcal C(y))$ and
\begin{equation*}
    \sraj(z)-\sraj(y) = a+ |\{r\in\mathcal C(y):r>a\}|.
\end{equation*}
\end{lem}

\begin{proof}
By Lemma~\ref{lem:diagram-recursion}, the set of columns with (invisible) dark clouds in row $2$ or below of $\snow(z)$ is $\iota_a(\mathcal C(y))$. If $F_a(\mathcal C(y)) \neq \emptyset$, then there exists a free cell in row $1$ of $\snow(z)$, thus we mark the rightmost free cell, which is in column $\iota_a(\widehat{F}_a(\mathcal C(y)))$. Furthermore, if $|F_a(\mathcal C(y))| \geq 2$, then there are at least two free cells, so we put an invisible dark cloud at $(1,1)$. Therefore, $\snow(z)$ has (invisible) dark clouds in columns $\Phi_a(\mathcal C(y))$.

By Lemma~\ref{lem:diagram-recursion} and Definition~\ref{def:snow-sraj}, $ \sraj(z)-\sraj(y)$ is exactly the number of cells that are not invisible dark clouds in the first row of $\snow(z)$, which is $a + |\{r\in\mathcal C(y):r>a\}|$, as $a$ is the number of cells from $\MD(z)$ originally and $|\{r\in\mathcal C(y):r>a\}|$ is the number of (invisible) dark clouds that will produce a snow cell in row $1$. 
\end{proof}

Recall that for $w \in S_n$, $\rho_w(t)$ is the length of a longest increasing subsequence of $w|_{>t}$. Such a longest increasing subsequence either uses $t+1$ or not, so $\rho_w(t) = \max\{L_w(t+1), \rho_w(t+1)\}$. If $\rho_w(r-1) \neq \rho_w(r)$, then $\rho_w(r-1) = \rho_w(r) +1$, and we call $r$ a \definition{drop} of $w$ in this case. Define the \definition{drop set} of $w$ as the set of drops of $w$:
\begin{align*}
     \mathcal R(w) &:= \{r\in[n]:\rho_w(r-1)=\rho_w(r)+1\},\\
    \mathcal R_t(w) &:= \{r\in \mathcal R(w):r >t\}.
\end{align*}

\begin{lem}
\label{lem:threshold-drop-count}
For $w \in S_n$ and $0 \leq t \leq n$, $ \rho_w(t) = |\mathcal R_t(w)|$.
\end{lem}

\begin{proof}
By definition, $\rho_w(r-1)-\rho_w(r) \in \{0,1\}$ and $\rho_w(r-1)-\rho_w(r) = 1$ if and only if $r \in  \mathcal R(w)$.
Since $\rho_w(n) = 0$, telescoping gives
\begin{align*}
    &\rho_w(t) = \sum_{r=t+1}^{n}
    \bigl(\rho_w(r-1)-\rho_w(r)\bigr) =
    |\{r\in\mathcal R(w):r>t\}| = |\mathcal R_t(w)|. \qedhere
\end{align*}
\end{proof}

We now define a recurrence on the drop sets. Our eventual goal is to show that the recurrences of $\mathcal C(z)$ and $\mathcal R(\omega(\code(z)))$ agree up to complement.

For \(C \subseteq [n]\), let $H_a(C) : = C \cap [a]$. If $H_a(C)\neq\varnothing$, let $\widehat{H}_a(C) := \max H_a(C)$. Define
\begin{equation*}
    \Psi_a(C):= \iota_a(C \setminus \{\widehat{H}_a(C)\}) \cup\{a+2\} \cup
    \begin{cases}
        \{1\},&|H_a(C)|\le1,\\
        \varnothing,&|H_a(C)|\geq 2, 
    \end{cases}
\end{equation*}
where $C \setminus \{\widehat{H}_a(C)\} = C$ if $H_a(C) =\varnothing$.

\begin{lem}
\label{lem:threshold-profile-recurrence}
Let $w$ be a regular paired word of $[n]$ and let $u=\TIns_a(w)$. Then
\begin{equation*}
    \mathcal R(u)=\Psi_a(\mathcal R(w)).
\end{equation*}
\end{lem}

\begin{proof}
Let $q = a+2$. By Lemma~\ref{lem:tight-insertion-exact}, $u$ is a regular paired word of $[n+2]$ and
\begin{align*}
    &L_u(\iota_a(r))=L_w(r) \quad \text{for }r \in [n],\\
    &L_u(1)=\rho_w(a)+2, \text{ and } L_u(q)=\rho_w(a)+1.
\end{align*}
It follows that
\begin{align*}
    \rho_u(0) &=\max\{\rho_w(0),\rho_w(a)+2\},\\
    \rho_u(t) &=\max\{\rho_w(t-1),\rho_w(a)+1\} &&\text{for } 1\leq t <  q,\\
    \rho_u(t)&=\rho_w(t-2) &&\text{for } q \leq t\leq n+2.
\end{align*}
To prove the statement, it suffices to show that 
\begin{align*}
    |\mathcal R_t(u)| = |\{r\in\Psi_a(\mathcal R(w)):r>t\}| \quad \text{for all }\,0 \leq t \leq n+2.
\end{align*}

If $t \geq q$, then 
\begin{align*}
     \{r\in\Psi_a(\mathcal R(w)):r>t\} = \iota_a(\mathcal R_{t-2}(w))
\end{align*}
so by Lemma~\ref{lem:threshold-drop-count},
\begin{align*}
    |\{r\in\Psi_a(\mathcal R(w)):r>t\}| =|\mathcal R_{t-2}(w)| =\rho_w(t-2) = \rho_u(t) =  |\mathcal R_t(u)|.
\end{align*}

Now suppose $1\le t < q$. Define $k_t:=|(\mathcal R(w)\cap[t,a])|$. Then by Lemma~\ref{lem:threshold-drop-count}, $\rho_w(t-1)=\rho_w(a)+k_t$. We split into cases based on whether $k_t$ is zero or not.

If $k_t = 0$, then $\rho_w(t-1)=\rho_w(a)$, so
\begin{align*}
    \rho_u(t) = \max\{\rho_w(t-1), \rho_w(a) +1\} = \rho_w(a) +1 = \rho_w(t-1) +1.
\end{align*}
Furthermore, since there are no drops in $[t,a]$,
\begin{align*}
    \{r\in\Psi_a(\mathcal R(w)):r>t\} = \iota_a(\mathcal R_{t-1}(w)) \cup \{q\},
\end{align*}
so by Lemma~\ref{lem:threshold-drop-count},
\begin{align*}
    |\{r\in\Psi_a(\mathcal R(w)):r>t\}| = |\mathcal R_{t-1}(w)| + 1 = \rho_w(t-1) + 1 = \rho_u(t) = |\mathcal{R}_t(u)|.
\end{align*}
If $k_t >0$, then $\rho_w(t-1)=\rho_w(a) + k_t$, so
\begin{align*}
    \rho_u(t) = \max\{\rho_w(t-1), \rho_w(a) + 1\} = \rho_w(t-1).
\end{align*}
Furthermore, since there exists at least one drop in $[t,a]$,
\begin{align*}
     \{r\in\Psi_a(\mathcal R(w)):r>t\} = \iota_a(\mathcal R_{t-1}(w) \setminus \{\widehat{H}_a(\mathcal{R}(w))\}) \cup \{q\},
\end{align*}
so by Lemma~\ref{lem:threshold-drop-count},
\begin{align*}
    |\{r\in\Psi_a(\mathcal R(w)):r>t\}| = |\mathcal R_{t-1}(w)| -1 + 1 = \rho_w(t-1) = \rho_u(t) = |\mathcal{R}_t(u)|.
\end{align*}

Finally, suppose $t=0$. By Lemma~\ref{lem:threshold-drop-count}, $\rho_w(0)=\rho_w(a)+ |H_a(\mathcal{R}(w))|$. We split into cases based on whether $|H_a(\mathcal{R}(w))| \leq 1$.

If $|H_a(\mathcal{R}(w))| \leq 1$, then $\rho_w(0) < \rho_w(a) +2$, so
\begin{align*}
    \rho_u(0) =\max\{\rho_w(0),\rho_w(a)+2\} = \rho_w(a)+2.
\end{align*}
Therefore, 
\begin{align*}
    \Psi_a(\mathcal{R}(w)) = \iota_a(\mathcal{R}(w) \setminus \{\widehat{H}_a(\mathcal{R}(w))\}) \cup\{a+2\} \cup \{1\},
\end{align*}
where $\mathcal{R}(w) \setminus \{\widehat{H}_a(\mathcal{R}(w))\} = \mathcal{R}_a(w)$. So by Lemma~\ref{lem:threshold-drop-count},
\begin{align*}
    |\Psi_a(\mathcal{R}(w))| = |\iota_a(\mathcal{R}_a(w))| + 2 = \rho_w(a) + 2 = \rho_u(0) = |\mathcal{R}(u)|.
\end{align*}
If $|H_a(\mathcal{R}(w))| \geq 2$, then $\rho_w(0) \geq \rho_w(a) +2$, so
\begin{align*}
     \rho_u(0) =\max\{\rho_w(0),\rho_w(a)+2\} = \rho_w(0).
\end{align*}
Therefore,
\begin{align*}
    \Psi_a(\mathcal{R}(w)) = \iota_a(\mathcal{R}(w) \setminus \{\widehat{H}_a(\mathcal{R}(w))\}) \cup\{a+2\}.
\end{align*}
So by Lemma~\ref{lem:threshold-drop-count},
\begin{align*}
    &|\Psi_a(\mathcal{R}(w))| = |\iota_a(\mathcal{R}(w))| -1 + 1 = \rho_w(0) = \rho_u(0) = |\mathcal{R}(u)|. \qedhere
\end{align*}
\end{proof}

We now show that the recurrences $\Phi_a$ and $\Psi_a$ are related by complement.

\begin{lem}
\label{lem:complement-intertwining}
Let $C \subseteq [n]$ and $0 \leq a \leq n$, then
\begin{align*}
    [n+2] \setminus \Phi_a(C) = \Psi_a([n]\setminus C).
\end{align*}
\end{lem}
\begin{proof}
Let $R=[n]\setminus C$ and $q = a+2$. Recall that $F_a(C)=[a]\setminus C$ and $H_a(R)=R\cap[a]$ by definition. Then we have 
\begin{align*}
    F_a(C) = [a]\setminus C = ([n]\setminus C)\cap[a] = R\cap[a] = H_a(R).
\end{align*}
Consequently, $|F_a(C)|=|H_a(R)|$ and $\widehat{F}_a(C)=\widehat{H}_a(R)$ whenever they are non-empty.

The statement is now equivalent to showing $\Phi_a(C) \sqcup \Psi_a(R) = [n+2]$. We recall the definitions of $\Phi_a(C)$ and $\Psi_a(R)$ as
\begin{align*}
     \Phi_a(C)
    :={}&\iota_a(C)  \, \cup \,
    \begin{cases}
        \{\iota_a(\widehat{F}_a(C))\},&F_a(C)\neq\varnothing,\\
        \varnothing,&F_a(C)=\varnothing,
    \end{cases}  \, \cup \, 
    \begin{cases}
        \{1\},&|F_a(C)|\ge2,\\
        \varnothing,&|F_a(C)|\le1.
    \end{cases}\\
    \Psi_a(R):= &\iota_a(R \setminus \{\widehat{H}_a(R)\}) \cup\{a+2\} \cup
    \begin{cases}
        \{1\},&|H_a(R)|\leq 1,\\
        \varnothing,&|H_a(R)|\geq 2.
    \end{cases}
\end{align*}
It is clear that $q$ is always in $\Psi_a(R)$ and never in $\Phi_a(C)$, and that $1$ is in exactly one of $\Psi_a(R)$ or $\Phi_a(C)$ as $|F_a(C)|=|H_a(R)|$. It remains to consider the letters in $[n+2] \setminus \{1,q\}$.

Let $r \in [n]$, then $\iota_a(r) \in \iota_a([n]) = [n+2] \setminus \{1,q\}$.
If $F_a(C) \neq \varnothing$ and $r = \widehat{F}_a(C)$, then $\iota_a(r) \in \Phi_a(C)$ and $\iota_a(r) = \iota_a(\widehat{H}_a(R)) \notin \Psi_a(R)$. Now suppose $F_a(C) = \varnothing$ or $ r \neq \widehat{F}_a(C)$, then
\begin{align*}
    r \in C, \, r \notin R \implies \iota_a(r) \in \Phi_a(C), \, \iota_a(r) \notin \Psi_a(R),\\
    r \notin C, \, r \in R \implies \iota_a(r) \notin \Phi_a(C), \, \iota_a(r) \in \Psi_a(R).
\end{align*}
Every element in $[n+2] \setminus \{1,q\}$ is $\iota_a(r)$ for some $r\in [n]$, so we are done.
\end{proof}

\begin{prop}
\label{p:cloud-drop-complement}
Let $z\in\ifpf_{2m}$, then
\begin{align*}
    \mathcal R\left(\omega(\code(z))\right) = [2m]\setminus\mathcal C(z).
\end{align*}
\end{prop}

\begin{proof}
We prove by induction on $m$. The base case for the empty fixed-point-free involution is immediate. Let $z = (a,y)$, $w = \omega(\code(y))$, and $u=\omega(\code(z))$. So $u=\TIns_a(w)$ by Definition~\ref{def:omega-recursion}. Then 
\begin{align*}
     \mathcal R(u) =\Psi_a(\mathcal R(w)) =\Psi_a([2m-2]\setminus\mathcal C(y)) =[2m]\setminus\Phi_a(\mathcal C(y)) =[2m]\setminus\mathcal C(z).
\end{align*}
The first equality is Lemma~\ref{lem:threshold-profile-recurrence}; the second equality is the inductive hypothesis; the third equality is Lemma~\ref{lem:complement-intertwining}, and the final equality is Lemma~\ref{lem:cloud-recurrence}.
\end{proof}

\begin{prop}
\label{prop:raj-equals-twice-sraj}
Let $z \in \ifpf_{2m}$, then
\begin{align*}
    \raj\bigl(\omega(\code(z))\bigr) = 2 \cdot \sraj(z).
\end{align*}
\end{prop}

\begin{proof}
We prove by induction on $m$. The base case for the empty fixed-point-free involution is immediate.  Let $z = (a,y)$, $w = \omega(\code(y))$, and $u=\omega(\code(z))$. So by Definition~\ref{def:omega-recursion}, $u=\TIns_a(w)$. By Lemma~\ref{lem:tight-insertion-exact}, $\mathcal J(u)-\mathcal J(w) = 2\rho_w(a)+3$.
Since $w \in S_{2m-2}$ and $u\in S_{2m}$, the difference of their Rajchgot indices is
\begin{align}
    \raj(u)-\raj(w) =
    \left[\binom{2m+1}{2}-\binom{2m-1}{2}\right]
    -\bigl(2\rho_w(a)+3\bigr)\notag =2\bigl(2m-2-\rho_w(a)\bigr).
\end{align}

On the other hand, Lemma~\ref{lem:cloud-recurrence} and Proposition~\ref{p:cloud-drop-complement} give
\begin{align*}
    \sraj(z)-\sraj(y) &= a+ |\{r\in\mathcal C(y):r>a\}|\\
     &= a+ (2m-2 - a) -|\{r\in\mathcal R(w):r>a\}|\\
     & = 2m-2 -\rho_w(a)
\end{align*}
where the last equality follows from
Lemma~\ref{lem:threshold-drop-count}.  

By the inductive hypothesis, $\raj(w) = 2\cdot\sraj(y)$, so
\begin{align*}
    &\raj(u) = \raj(w)+2\bigl(\sraj(z)-\sraj(y)\bigr) = 2\sraj(z).  \qedhere
\end{align*}
\end{proof}

Finally, we can prove Theorem~\ref{t: main_degree} and Corollary~\ref{cor: Cm-reg}.
\begin{proof}[Proof of Theorem~\ref{t: main_degree}]
Let $w = \omega(\code(z))^{-1}$, then
    \begin{align*}
        \deg \fpfG_z(\x) = \deg \fG_w(\x) = \raj(w^{-1}) = 2 \cdot \sraj(z)
    \end{align*}
The first equality is Theorem~\ref{t: main_top}; the second equality is Theorem~\ref{t: psw raj}; and the final equality is Proposition~\ref{prop:raj-equals-twice-sraj}.
\end{proof}

\begin{proof}[Proof of Corollary~\ref{cor: Cm-reg}]
Since $X_z^{\mathrm{SS}}$ is Cohen--Macaulay and $\operatorname{ht}(I^\mathrm{SS}_z) = \ell_{\mathrm{FPF}}(z)$,
\begin{align*}
    \reg(X_z^{\mathrm{SS}}) = \deg_t K_z^{\mathrm{std}}(t) - \operatorname{ht}(I^\mathrm{SS}_z) = \frac{1}{2} \cdot \deg_{\x} \fpfG_z(\x) - \ell_{\mathrm{FPF}}(z) = \sraj(z) - \ell_{\mathrm{FPF}}(z).
\end{align*}
The first equality is Lemma~\ref{lem:CM-regularity-K-polynomial}; the second equality is Lemma~\ref{lem:standard-torus-degree-factor}; and the final equality is Theorem~\ref{t: main_degree}.
\end{proof}

\section{Maximal Castelnuovo--Mumford Regularity}
\label{s: max cm}
The goal of this section is to compute the maximal Castelnuovo--Mumford regularity of skew-symmetric matrix Schubert varieties indexed by $\ifpf_{2m}$. Pechenik, Speyer, and Weigandt answered this question for matrix Schubert varieties indexed by $S_n$.
\begin{thm}~\cite{psw24}*{Theorem~1.6}
\label{t: maximal CM-Reg for MSV}
    For $n \in \ZZ_{>0}$, define $k$ by $\binom{k}{2} \leq n < \binom{k+1}{2}$. Then
    \begin{align*}
        \max_{w \in S_n} \reg(X_w) = \binom{n+1}{2} - kn + \binom{k+1}{3}.
    \end{align*}
\end{thm}
A permutation is \definition{layered} if it is the longest element in $S_\alpha = S_{\alpha_1} \times S_{\alpha_2} \times \cdots \times S_{\alpha_r}$ for some composition $\alpha$. Pechenik, Speyer, and Weigandt identified the maximizers in Theorem~\ref{t: maximal CM-Reg for MSV} as matrix Schubert varieties labeled by certain layered permutations~\cite{psw24}*{Theorem~5.9}.

Our strategy is to optimize regularity in two stages. We first fix
a level sequence and minimize the fpf-Coxeter length. We then optimize over level sequences to obtain a global maximum.

Recall that for a regular paired word $w = (b_1,c_1) \dots (b_m, c_m) \in S_{2m}$, its level sequence is $(h_1, \dots, h_m)$ where $h_1 \geq \cdots \geq h_m = 2$ and $L_w(b_i) = L_w(c_i) +1 =  h_i$.
\begin{lem}
\label{l: is a level sequence}
    An integral sequence $V = (h_1, \dots, h_m)$ is the level sequence of some regular paired word $w \in S_{2m}$ if and only if 
    \begin{align}
    \label{eq: what makes level sequence}
        \begin{cases}
             h_1 \geq h_2 \geq \cdots \geq h_m = 2, \text{ and}\\
             0 \leq h_i - h_{i+1} \leq 2.
        \end{cases}
    \end{align}
\end{lem}
\begin{proof}
    The forward direction is immediate from definition. Conversely, suppose \eqref{eq: what makes level sequence} holds. We induct on $m$. The base case when $m=1$ is trivial. For $m>1$, the induction hypothesis gives a regular paired word $u \in S_{2m-2}$ with level sequence $(h_2,\ldots,h_m)$. Then $\rho_u(0)=h_2$ and $\rho_u(2m-2) = 0$. Since $ 0\leq h_1-2\leq h_2$ and each successive drop of \(\rho_u\) has size at most one, it attains every integer between \(h_2\) and \(0\). So there exists $a\in\{0,\ldots,2m-2\}$ such that $\rho_u(a)=h_1-2$. By Lemma~\ref{lem:tight-insertion-exact}, the word $w: = \TIns_a(u)$ is regular, and its level sequence is obtained by inserting $h_1$ into $(h_2,\ldots,h_m)$ which is exactly $V$.
\end{proof}

\begin{defn}
 For a level sequence $V = (h_1, \dots, h_m)$, define $\nu_h(V) := |\{i: h_i = h\}|$ to be its level counts. 
\end{defn}

\begin{lem}
\label{l: AHHHHHH}
    Let $z = (a,y) \in \ifpf_{2m}$ and let  $V = (h_1, \dots, h_{m-1})$ be the level sequence of $w := \omega(\code(y))$. Let $H := \rho_w(a) +2$, then
    \begin{align*}
        A_H(z) := \{r \in [2m-2] : L_w(r) \geq H-1\} \subseteq [a],
    \end{align*}
    and 
    \begin{equation}
    \label{eq: AH size}
        |A_H(z)| = \nu_{H-1}(V) + 2 \sum_{h\geq H}^{2m} \nu_h(V).
    \end{equation}
\end{lem}
\begin{proof}
    For every $r \in [2m-2]$ such that $L_w(r) \geq H-1$, we have $r \leq a$. Otherwise, $r >a $ would imply that $\rho_w(a) \geq L_w(r) \geq H-1 = \rho_w(a) +1$, which is a contradiction. So the first statement holds.

    Let $w = (b_1,c_1) \dots (b_{m-1},c_{m-1})$. Since $w$ is regular, a pair $(b_i,c_i)$ of level $h$ satisfies $L_w(b_i) = h$ and $L_w(c_i) = h-1$. Therefore, if $h \geq H$, then $b_i,c_i \in A_H(z)$; if $h = H-1$, then $b_i \in A_H(z)$ and $c_i \notin A_H(z)$; and if $ h < H-1$, then $b_i, c_i \notin A_H(z)$. Hence the second statement follows.
\end{proof}

\begin{lem}
\label{l: coxeter length lower bound}
Let $z \in \ifpf_{2m}$ and let $V = (h_1, \dots, h_m)$ be the level sequence of $\omega(\code(z))$. Then
    \begin{align*}
        B(V) := \sum_{h = 0}^{2m} \left( 2\binom{\nu_h(V)}{2} + \nu_h(V) \nu_{h+1}(V) \right) \leq  \ell_{\fpf}(z).
    \end{align*}
\end{lem}
\begin{proof}
    We prove by induction on $m$, the base case when $m=1$ is trivial. Let $z = (a,y)$, $w = \omega(\code(y))$, $u = \omega(\code(z)) = \TIns_a(w)$, and $H = \rho_w(a)+2$. Let $V'$ be the level sequence of $w$.
    By Lemma~\ref{lem:tight-insertion-exact},
    \begin{align*}
        \nu_h(V) = 
        \begin{cases}
              \nu_h(V') +1 & \text{if } h = H, \\
              \nu_h(V')  & \text{if } h \neq H.
            \end{cases}
    \end{align*}
        Thus, the difference between $B(V)$ and $B(V')$ is giving by the terms that involve $\nu_H(V')$.
    \begin{align}
    \label{eq: lower bound difference}
        B(V) - B(V') &= 2 \left( \binom{\nu_H(V')+1}{2} - \binom{\nu_H(V')}{2} \right) + \nu_{H-1}(V') + \nu_{H+1}(V')\\
        &= 2 \nu_H(V') + \nu_{H-1}(V') + \nu_{H+1}(V'). \notag
    \end{align}
    Then by Lemma~\ref{l: AHHHHHH} and~\eqref{eq: lower bound difference},
    \begin{align*}
        a \geq |A_H(z)| = \nu_{H-1}(V') + 2 \sum_{h\geq H}^{2m} \nu_h(V') \geq 2\nu_{H}(V') + \nu_{H-1}(V') + \nu_{H+1}(V') = B(V) - B(V').
    \end{align*}
    Finally, by the inductive hypothesis, we have $\ell_{\fpf}(y) \geq B(V')$, so
    \begin{align*}
        &\ell_{\fpf}(z) = \ell_{\fpf}(y) + a \geq B(V') + a \geq B(V). \qedhere
    \end{align*}
\end{proof}

Among all fixed-point-free involutions whose maximal inverse Hecke atom has a given level sequence, we can construct a unique representative that minimizes fpf-Coxeter length. 
\begin{defn}
\label{d: zv}
    Given a level sequence $V = (h_1, h_2,  \dots, h_m)$, let 
    \begin{align*}
        a_i := 2\cdot |\{j>i: h_j = h_i\}| + |\{j > i: h_j = h_i -1\}| \quad \text{for }i \in [m],
    \end{align*}
    and let $z_V$ be the fixed-point-free involution in $\ifpf_{2m}$ with $\code(z_V) = (a_1, \dots, a_m)$. 
\end{defn}

Since each $a_i$ is bounded above by $2(m-i)$, $z_V$ is well-defined by Lemma~\ref{l: code composition bijection}.

\begin{lem}
    \label{l: zv has given level sequence and decreasing LIS}
    Let $V = (h_1, h_2,  \dots, h_m)$ be a level sequence and let $u = \omega(\code(z_V))$. Then $u$ has level sequence $V$, and
     \begin{align*}
         L_u(1) \geq L_u(2) \geq \cdots \geq L_u(2m).
     \end{align*}
\end{lem}
\begin{proof}
    We induct on $m$, the base case when $m=1$ is trivial.

    Let $z_V = (a_1,y)$, $w = \omega(\code(y))$, so $u = \TIns_{a_1}(w)$. Since $\code(y) = (a_2, \dots, a_m)$ is constructed from the level sequence $V' := (h_2, \dots, h_m)$, we have $y = z_{V'}$, so by the inductive hypothesis,
    \begin{align}
    \label{eq:LIS mono-decrease}
        L_w(1) \geq L_w(2) \geq \cdots \geq L_w(2m-2).
    \end{align}

    Let $ a' := \min\{t\in \{0,\dots,2m-2\} : \rho_w(t) = h_1-2\}$. It is well defined as $0\leq h_1 - 2 \leq h_2$ and $\rho_w(i)-1 \leq \rho_w(i+1) \leq \rho_w(i)$. Our goal is to show that $a'=a_1$. Let $z':=(a',y)$. By construction, $h_1 =\rho_w(a')+2$, so by Lemma~\ref{l: AHHHHHH}, $A_{h_1}(z')\subseteq[a']$. We claim that equality holds. If $a'= 0$, then the claim is trivial. If $a'>0$, minimality of $a'$ implies that $\rho_w(a'-1)=h_1-1$. Since $ \rho_w(a'-1)=\max\{L_w(a'),\rho_w(a')\}$ and $\rho_w(a')=h_1-2$, we get that $\rho_w(a'-1) = L_w(a')=h_1-1$. Therefore, by~\eqref{eq:LIS mono-decrease}, we have $L_w(r) \geq h_1 - 1$ for all $ r \leq a'$, so the claim follows.

    Since $h_1 \geq h_i$ for all $ i \in [m]$, by Lemma~\ref{l: AHHHHHH},
    \begin{align*}
        a' &=|A_{h_1}(z')| =\nu_{h_1-1}(V') + 2\sum_{h\geq h_1}^{2m}\nu_h(V') =\nu_{h_1-1}(V')+2\nu_{h_1}(V')\\
        &= \bigl|\{j>1:h_j=h_1-1\}\bigr|
      +2\bigl|\{j>1:h_j=h_1\}\bigr| =a_1.
    \end{align*}
Thus $\rho_w(a_1)=h_1-2$. By Lemma~\ref{lem:tight-insertion-exact}, the level sequence of $u$ is obtained by inserting $h_1$ into $V'$, which is exactly $V$. Lemma~\ref{lem:tight-insertion-exact} also gives
\begin{align}
\label{eq: final weakly decreasing LIS}
        (L_u(1), \dots, L_u(2m)) = (h_1, L_w(1), \dots, L_w(a_1), h_1-1, L_w(a_1+1), \dots, L_w(2m-2)).
\end{align}
We have  $L_w(1) \leq h_2 \leq h_1$. Since $A_{h_1}(z')=[a'] = [a_1]$, the entries $L_w(1), \dots, L_w(a_1)$ are at least $h_1-1$. Since $\rho_w(a')=h_1-2$, the entries $L_w(a_1+1), \dots, L_w(2m-2)$ are at most $h_1 -2$. Hence,~\eqref{eq: final weakly decreasing LIS} is a weakly decreasing sequence.
\end{proof}

\begin{lem}
    \label{l: zv unique min}
    Let $z \in \ifpf_{2m}$ and $V= (h_1, \dots, h_m)$ be the level sequence of $ u = \omega(\code(z))$. Then $\ell_{\fpf}(z_V) \leq \ell_{\fpf}(z)$ with equality if and only if $z_V = z$.
\end{lem}
\begin{proof}
    Recall that $\code(z_V) = (a_1, \dots, a_m)$ where 
    \begin{align*}
          a_i = 2\cdot |\{j>i: h_j = h_i\}| + |\{j > i: h_j = h_i -1\}|.
    \end{align*}
    So by construction,
    \begin{align*}
        \ell_{\fpf}(z_V) = \sum_{i=1}^m a_i = 2 \sum_{h \geq 0} \binom{\nu_h(V)}{2} + \sum_{h \geq 0} \nu_h(V) \nu_{h+1} (V) = B(V),
    \end{align*}
    so $z_V$ achieves the lower bound in Lemma~\ref{l: coxeter length lower bound}. We now show that no other elements do.

    We induct on $m$, the base case when $m=1$ is trivial. Suppose $z \in \ifpf_{2m}$ satisfy $u:= \omega(\code(z))$ has level sequence $V$ and $\ell_{\fpf}(z) = B(V)$. Let $z= (a,y)$, $w = \omega(\code(y))$, and $ H = \rho_w(a) +2$. Let $V'$ be the level sequence $w$, then by~\eqref{eq: AH size} and~\eqref{eq: lower bound difference},
    \begin{align*}
        \ell_{\fpf}(z) - B(V) =  
        \left( \ell_{\fpf}(y) - B(V') \right) + \left( a - |A_H(z)| \right) + \nu_{H+1}(V') + 2\sum_{h \geq H + 2} \nu_h(V').
    \end{align*}
    On the RHS, the first term is nonnegative by Lemma~\ref{l: coxeter length lower bound}, the second term is nonnegative by Lemma~\ref{l: AHHHHHH}, and the remaining terms are nonnegative by definition. Since we assumed $\ell_{\fpf}(z) = B(V)$, all the terms on the RHS are identically zero. Every level in \(V'\) is at most \(H\). Since \(V\) is obtained by inserting \(H\) into \(V'\), its largest entry is \(H\). Thus \(H=h_1\) and \(V'=(h_2,\ldots,h_m)\). Also, $\ell_{\fpf}(y) = B(V')$ implies that $ y = z_{V'}$ by the inductive hypothesis. Finally, since $a = |A_H(z)|$ and $H = h_1 \geq h_i$ for $i \in \{2, \dots,m\}$, we have
    \begin{align*}
        a = |A_H(z)| = 2 \nu_{H}(V') + \nu_{H-1}(V') = 2\cdot |\{j>1: h_j = h_1\}| + |\{j > 1: h_j = h_1 -1\}| = a_1.
    \end{align*}
    Therefore, $ z = (a,y) = (a_1,z_{V'}) = z_V$, and we prove uniqueness. 
\end{proof}
The fixed-point-free involution \(z_V\) uniquely maximizes Castelnuovo--Mumford regularity among those whose maximal inverse Hecke atom has level sequence \(V\).
\begin{lem}
\label{l: zv maximize CM-Reg}
     Let $z \in \ifpf_{2m}$ and $V= (h_1, \dots, h_m)$ be the level sequence of $\omega(\code(z))$, then $\reg(X^\mathrm{SS}_z) \leq \reg(X^\mathrm{SS}_{z_V})$ with equality if and only if $z_V = z$.
\end{lem}
\begin{proof}
    Let $u = \omega(\code(z))$, then 
    \begin{align*}
        \mJ(u) = \sum_{i=1}^m (2h_i-1) = 2 \sum_{i=1}^m h_i - m.
    \end{align*}
    By Proposition~\ref{prop:raj-equals-twice-sraj},
    \begin{align*}
        \sraj(z) = \frac{1}{2}\raj(u) = m(m+1) - \sum_{i=1}^m h_i.
    \end{align*}
    Hence by Corollary~\ref{cor: Cm-reg} and Lemma~\ref{l: zv has given level sequence and decreasing LIS},
    \begin{align*}
        \reg(X^\mathrm{SS}_z) &= \sraj(z) - \ell_{\fpf}(z) = m(m+1) - \sum_{i=1}^m h_i - \ell_{\fpf}(z), \text{ and }\\
        \reg(X^\mathrm{SS}_{z_V}) &= \sraj(z_V) - \ell_{\fpf}(z_V) = m(m+1) - \sum_{i=1}^m h_i - \ell_{\fpf}(z_V).
    \end{align*}
    Then the statement follows from Lemma~\ref{l: zv unique min}.
\end{proof}

In particular, the proof of Lemma~\ref{l: zv maximize CM-Reg} implies that the problem of finding the maximal Castelnuovo--Mumford regularity reduces to finding a level sequence $V=(h_1, \dots, h_m)$ that minimizes the following quantity.
\begin{align}
\label{eq: ev def}
    \varepsilon(V) :=& \sum_{i=1}^m h_i + \ell_{\fpf}(z_V) = \sum_{i=1}^m h_i + B(V)\\  \notag
    =& \sum_{i=1}^m h_i + 2 \sum_{h = 0}^{2m} \binom{\nu_h(V)}{2} + \sum_{h = 0}^{2m} \nu_h(V) \nu_{h+1}(V)\\ \notag
    =& \sum_{h = 0}^{2m} h \cdot \nu_h(V) + 2 \sum_{h = 0}^{2m} \binom{\nu_h(V)}{2} + \sum_{h = 0}^{2m} \nu_h(V) \nu_{h+1}(V)\\ \notag
    =& \sum_{h = 0}^{2m} \left( \nu_h(V)^2 + (h-1) \nu_h(V) \right) + \sum_{h = 0}^{2m} \nu_h(V) \nu_{h+1}(V). \notag
\end{align}

We optimize $\varepsilon(V)$ over level sequences by
expressing it as the average of two quadratic functions.
\begin{defn}
    Let $V = (h_1, \dots, h_m)$ be a level sequence. We group adjacent level multiplicities in the following two ways. For $i \geq 1$, define
    \begin{align*}
    p_i(V) := \nu_{2i}(V) + \nu_{2i+1}(V) \quad &\text{and} \quad q_i(V) := \nu_{2i-1}(V) + \nu_{2i}(V);\\
    p(V) := (p_1(V),p_2(V), \dots) \quad &\text{and} \quad q(V) := (q_1(V),q_2(V), \dots).
    \end{align*}
    Since $\nu_1(V) = 0$ for any level sequence, $\sum_{i \geq 1}p_i = \sum_{i \geq 1}q_i = m$.
    
    Furthermore, for a finitely supported nonnegative integral sequence $v = (v_1, v_2, \dots)$, define
    \begin{align*}
        \mathcal{F}(v) := \sum_{i \geq 1}(v_i^2 + (2i-1)v_i),
    \end{align*}
    and define the minimum of $\mathcal{F}(v)$ across sequences that sum to $m$ as
    \begin{align*}
        f_m := \min_{\sum_{i} v_i = m} \mathcal{F}(v).
    \end{align*}
\end{defn}

\begin{lem}
\label{l: ev = p + q}
    Let $V = (h_1, \dots, h_m)$ be a level sequence, then 
    \begin{align*}
        \varepsilon(V) = \frac{\mathcal{F}(p(V)) +\mathcal{F}(q(V))}{2}.
    \end{align*}
    In particular, $\varepsilon(V) \geq f_m$ with equality if and only if $\mathcal{F}(p(V)) = \mathcal{F}(q(V)) = f_m$.
\end{lem}
\begin{proof}
    The quadratic terms give 
    \begin{align}
    \label{eq: quadratic term}
        \frac{1}{2}\sum_{i \geq 1} (p_i(V)^2 + q_i(V)^2) = \sum_{h = 0}^{2m} \nu_h(V)^2 + \sum_{h=0}^{2m} \nu_h(V) \nu_{h+1}(V).
    \end{align}
    Each $\nu_h(V)^2$ appears once in both $\mathcal{F}(p(V))$ and $\mathcal{F}(q(V))$. And each $\nu_h(V) \nu_{h+1}(V)$ appears either in $\mathcal{F}(p(V))$ or $\mathcal{F}(q(V))$ with multiplicity two.

    The linear terms give 
    \begin{align}
    \label{eq: linear term}
        \frac{1}{2} \sum_{i \geq 1} (2i-1)(p_i + q_i) = \sum_{h =0}^{2m} (h-1)\nu_h(V).
    \end{align}
    If $h$ is even, then $h = 2i$ for some $i$ and $\nu_h(V)$ has coefficient $2i-1$ in both $\mathcal{F}(p(V))$ and $\mathcal{F}(q(V))$. If $h$ is odd, then $h =2i+1$ for some $i$ and $\nu_h(V)$ has coefficient $2i-1$ and $2i+1$ in $\mathcal{F}(p(V))$ and $\mathcal{F}(q(V))$ respectively.

    The equality in the statement then follows from summing~\eqref{eq: quadratic term} and~\eqref{eq: linear term}, and the second statement is immediate from the equality.
\end{proof}

\begin{defn}
\label{d :zs}
    For $m \in \ZZ_{>0}$, define $k$ by $\binom{k}{2} \leq m < \binom{k+1}{2}$. Let $r := m - \binom{k}{2}$, so $0 \leq  r < \binom{k+1}{2} - \binom{k}{2} = k$. For any subset $S \subseteq [k]$ with $|S| = r$, define
    \begin{align*}
        v_i = \begin{cases}
        k-i+1 & \text{if } 1 \leq i \leq k \text{ and }i \in S, \\
        k-i & \text{if } 1 \leq i \leq k \text{ and }i \notin S, \\
        0  & \text{if } i >k.
    \end{cases}
    \end{align*}
    Then $v_S:= (v_1, v_2, \dots)$ is a finitely supported nonnegative integral sequence whose entries sum to $m$.

    Now define 
    \begin{align*}
        V_S = \bigl( \underbrace{2k,\ldots,2k}_{v_k},
        \underbrace{2k-2,\ldots,2k-2}_{v_{k-1}},
        \ldots,
        \underbrace{4,\ldots,4}_{v_2},
        \underbrace{2,\ldots,2}_{v_1}
        \bigr).
        \end{align*}
        Equivalently, $V_S$ is the unique weakly decreasing sequence of length $m$ satisfying
        \begin{align*}
            \nu_{2i}(V_S)=v_i
            \quad\text{and}\quad
            \nu_{2i+1}(V_S)=0
            \qquad\text{for all }i\geq 1.
        \end{align*}
            It is immediate from the definition that $p(V_S) = q(V_S) = v_S$ as all the odd level counts are zero. By Lemma~\ref{l: is a level sequence}, $V_S$ is a level sequence as it satisfies~\eqref{eq: what makes level sequence}. Let $z_S := z_{V_S}$ be the fixed-point-free involution in $\ifpf_{2m}$ such that $\omega(\code(z_S))$ has level sequence $V_S$.
    \end{defn}

    We now show that this construction achieves the lower bound $f_m$.

\begin{lem}
\label{l: everything equal}
    For $m \in \ZZ_{>0}$, define $k$ by $\binom{k}{2} \leq m < \binom{k+1}{2}$.  Then for any $S \subseteq [k]$ with $|S| = m - \binom{k}{2}$, we have
    \begin{align*}
        \varepsilon(V_S) = \mathcal{F}(v_S) =  f_m = 2km - 2 \binom{k+1}{3}.
    \end{align*}
\end{lem}
\begin{proof}
    We first show that for a finitely supported nonnegative integral sequence $v = (v_1, v_2, \dots)$ whose entries sum to $m$, $\mathcal{F}(v) \geq 2km - 2 \binom{k+1}{3}$. Since the entries of $v$ sum to $m$,
    \begin{align}
    \label{eq: fv - 2km}
        \mathcal{F}(v) - 2km &= \sum_{i \geq 1}(v_i^2 + (2i-1)v_i) - 2k \sum_{i \geq 1} v_i\\
        &= \sum_{i \geq 1}(v_i^2 + (2i-1 - 2k)v_i). \notag
    \end{align}
    
    If $1 \leq i \leq k$, then 
    \begin{align}
    \label{eq: i <= k}
        v_i^2 + (2i-1 - 2k)v_i = (v_i - k+i)(v_i -k+i-1) - (k-i)(k-i+1).
    \end{align}
    In particular, $(v_i - k+i)(v_i -k+i-1) \geq 0$ with equality if and only if $ v_i = k-i$ or $v_i = k-i+1$.

    If $ k< i$, then 
    \begin{align}
    \label{eq: i > k}
        v_i^2 + (2i-1 -2k)v_i = v_i(v_i + 2(i-k)-1) \geq 0.
    \end{align}
    Since $2(i-k)-1 > 0$ and $v_i \geq 0$, equality holds if and only if $v_i = 0$.

    Therefore, substituting~\eqref{eq: i <= k} and~\eqref{eq: i > k} into~\eqref{eq: fv - 2km} gives
    \begin{align*}
        \mathcal{F}(v) - 2km  \geq  - \sum_{i=1}^k (k-i)(k-i+1) = - 2 \binom{k+1}{3},
    \end{align*}
    with equality if and only if
    \begin{align}
    \label{eq: equality condition for vi}
        \begin{cases}
        v_i = k-i \text{ or } v_i = k-i+1 & \text{for } 1 \leq i \leq k, \text{ and}\\
        v_i = 0 & \text{for } i > k.
    \end{cases}
    \end{align}
    Thus $f_m \geq 2km - 2 \binom{k+1}{3}$. It remains to show that $v_S$ achieves the lower bound. 
    
    Since $v_S = p(V_S) = q(V_S)$, we have $\varepsilon(V_S) = \mathcal{F}(v_S)$. Finally, since $v_S$ satisfies~\eqref{eq: equality condition for vi},
    \begin{align*}
        &\varepsilon(V_S) = \mathcal{F}(v_S) = 2km - 2 \binom{k+1}{3} \leq f_m \leq \mathcal{F}(v_S). \qedhere
    \end{align*}
\end{proof}

\begin{proof}[Proof of Theorem~\ref{t: maximize fpf reg}]
     For any $z \in \ifpf_{2m}$, let $V$ be the level sequence of $\omega(\code(z))$, then
    \begin{align*}
        \reg(X^{\mathrm{SS}}_z) &\leq \reg(X^{\mathrm{SS}}_{z_V}) && (\text{Lemma}~\ref{l: zv maximize CM-Reg})\\
        &= m(m+1) - \varepsilon(V) && (\text{Corollary}~\ref{cor: Cm-reg}\text{ and Lemma}~\ref{l: zv has given level sequence and decreasing LIS})\\
        &\leq m(m+1) - f_m && (\text{Lemma}~\ref{l: ev = p + q})\\
        &= m(m+1)-  2km + 2 \binom{k+1}{3}. && (\text{Lemma}~\ref{l: everything equal})
    \end{align*}
    Since $z$ was arbitrary, 
    \begin{align*}
        \max_{z \in \ifpf_{2m}} \reg(X^{\mathrm{SS}}_z) \leq m(m+1) -  2km + 2 \binom{k+1}{3}.
    \end{align*}

    On the other hand, let $S \subseteq [k]$ be any subset such that $|S| = m -\binom{k}{2}$. By Lemma~\ref{l: everything equal} and Corollary~\ref{cor: Cm-reg},
    \begin{align*}
         \max_{z \in \ifpf_{2m}} \reg(X^{\mathrm{SS}}_z) \geq \reg(X^{\mathrm{SS}}_{z_S}) = m(m+1) - \varepsilon(V_S) = m(m+1)-  2km + 2 \binom{k+1}{3}.
    \end{align*}
    Therefore, combining the two inequalities gives
    \begin{align*}
        \max_{z \in \ifpf_{2m}} \reg(X^{\mathrm{SS}}_z) = m(m+1)-  2km + 2 \binom{k+1}{3} = 2 \cdot \max_{w \in S_m} \reg(X_w),
    \end{align*}
    where the second equality is Theorem~\ref{t: maximal CM-Reg for MSV}.
\end{proof}

\begin{rem}
\label{rem:nonlayered-regularity-maximizers}
Let  $ \mathbf{d}(m):= (2m-2, 2m-4, \cdots, 0)$
be the match code of the longest element in $\ifpf_{2m}$. Then since $V_S$ has only even levels,
\begin{align*}
    \code(z_S) = \mathbf{d}(v_k)\mathbf{d}(v_{k-1}) \cdots \mathbf{d}(v_1),
\end{align*}
so the fixed-point-free involutions $z_S$ constructed in Definition~\ref{d :zs} are layered. However, there exist non-layered fixed-point-free involutions that achieve the maximal Castelnuovo--Mumford regularity in Theorem~\ref{t: maximize fpf reg}.
\end{rem}
\begin{exa}
Let $  z=(1,3)(2,4)(5,8)(6,7)\in\ifpf_8$. Then $\code(z)=(1,0,2,0)$ and $ \omega(\code(z))=(1,3)(2,4)(6,8)(5,7)$ has level sequence $V=(5,4,2,2)$. Therefore, $\sraj(z)=4\cdot5-(5+4+2+2)=7$ and $ \ell_{\fpf}(z)=3$. So by Corollary~\ref{cor: Cm-reg}, $\reg (X_z^{\mathrm{SS}})=4$. This is the maximum in $\ifpf_8$ by Theorem~\ref{t: maximize fpf reg}, but $z$ is not layered.
\end{exa}

\appendix
\section{Degrees of Involution Grothendieck polynomials}
\label{app:involution-grothendieck}

Another family of polynomials that frequently appear together with the symplectic Grothendieck polynomials is the \definition{involution Grothendieck polynomials}. They are indexed by involutions $\mI_n:=\{z\in S_n:z=z^{-1}\}$. The involution Grothendieck polynomials $\invG_z(\x)$ are combinatorial $K$-theoretic analogues of involution
Schubert polynomials~\cite{MP21Principal}*{Definition~4.2}. Unlike the
symplectic Grothendieck polynomials, they have no known general direct
geometric interpretation and are generally different from the \definition{orthogonal
Grothendieck polynomials} representing the K-theory classes of orthogonal orbit closures (see~\cite{MWGrothExp}*{\S2.4}). Regardless, we show that the degrees of involution Grothendieck polynomials agree with those of ordinary Grothendieck polynomials indexed by the same involutions.

Let $\circ$ denote the \definition{Demazure product}, the associative
product on $S_n$ determined by
\[
    w\circ s_i=
    \begin{cases}
        ws_i,&\ell(ws_i)>\ell(w),\\
        w,&\ell(ws_i)<\ell(w).
    \end{cases}
\]
In particular, $s_i\circ s_i=s_i$, and inversion reverses the product:
$(u\circ v)^{-1}=v^{-1}\circ u^{-1}$.
For $z\in\mI_n$, define its set of \definition{involution Hecke atoms} and
its \definition{involution length} by
\[
    \mathcal B_{\inv}(z):=\{w\in S_n:w^{-1}\circ w=z\},
    \qquad
    \ell_{\inv}(z):=\min_{w\in\mathcal B_{\inv}(z)}\ell(w).
\]
The \definition{involution Grothendieck polynomial} of $z$ is
\[
    \invG_z(\x)
    :=\sum_{w\in\mathcal B_{\inv}(z)}
       (-1)^{\ell(w)-\ell_{\inv}(z)}\fG_w(\x).
\]
Its lowest-degree homogeneous component is the involution Schubert polynomial $\invS_z(\x)$, of degree $\ell_{\inv}(z)$.

Let $ \Delta_n:=\{(i,j)\in\ZZ_{>0}^2:i+j\leq n\}$. For a diagram $P\subseteq\Delta_n$, label its cell $(i,j)$ by
$s_{i+j-1}$, and let $\delta(P)$ be the Demazure product of these labels,
reading rows from top to bottom and each row from right to left.
Let
\[
    \PD(w):=\{P\subseteq\Delta_n:\delta(P)=w\}
\]
be the set of pipe dreams of $w$. The row and column
weights $\rwt(P),\cwt(P)\in\ZZ_{\geq0}^n$ record the numbers of cells in
each row and column, respectively. The set $\PD(w)$ combinatorially computes ordinary Grothendieck polynomials,
\begin{equation}
\label{eq:appendix-ordinary-pipe-dream}
    \fG_w(\x)
    =\sum_{P\in\PD(w)}(-1)^{|P|-\ell(w)} \x^{\rwt(P)}.
\end{equation}

For $z\in\mI_n$, its \definition{involution Hecke pipe dreams} are
\[
    \mathcal D^{\inv}(z)
    :=\{D\subseteq\Delta_n:
       D\subseteq\{(i,j):i\geq j\},
       \delta(D)^{-1}\circ\delta(D)=z\}.
\]
This is the Hecke-word definition of the diagrams in
\cite{MP21Principal}*{\S4}. Marberg and Pawlowski proved that
\begin{equation}
\label{eq:appendix-involution-pipe-dream}
    \invG_z(\x)
    =\sum_{D\in\mathcal D^{\inv}(z)}
      (-1)^{|D|-\ell_{\inv}(z)}
      \prod_{(i,i)\in D}x_i
      \prod_{\substack{(i,j)\in D\\i>j}}
      (x_i+x_j-x_ix_j).
\end{equation}
We relate the involution Hecke pipe dreams to pipe dreams in $\PD(z)$ that are symmetric. For a diagram $D$ weakly below the diagonal, let
$D^{\mathrm{sym}}:=D\cup D^{\mathsf T}$ be its \definition{symmetric completion}.

\begin{lem}
\label{lem:appendix-symmetric-completion}
For $D\subseteq\Delta_n$ weakly below the diagonal,
\[
    \delta(D^{\mathrm{sym}})=\delta(D)^{-1}\circ\delta(D).
\]
\end{lem}

\begin{proof}
We induct on $n$. The base cases  when $n\leq2$ are immediate. Let $  E:=\{(i,j)\in D:i,j\geq2\}$ and $u:=\delta(E)$. Let $\epsilon=1$ if $(1,1)\in D$ and $\epsilon=0$ otherwise. Let
\[
    \{r>1:(r,1)\in D\}=\{r_1<\cdots<r_t\}
    \quad \text{and}\quad  v:=s_{r_1}\circ\cdots\circ s_{r_t}.
\]
For $(r,1)$ in the first column, $s_r$ commutes with every interior label appearing later in the reading word since any later
interior cell $(i,j)$ satisfies $i>r$ and $j\geq2$, so
$i+j-1\geq r+2$. Therefore, we can move the labels in the first column to the
end, preserving their relative order, to obtain
\[
    \delta(D)=s_1^\epsilon\circ u\circ v.
\]
In $D^{\mathrm{sym}}$, the reflected cells from the first column form the first
row and are read in reverse order. So the reading word of the symmetric completion is
\begin{equation}
\label{eq:appendix-completion-factorization}
    \delta(D^{\mathrm{sym}})
    =v^{-1}\circ s_1^\epsilon
       \circ\delta(E^{\mathrm{sym}})\circ v.
\end{equation}

Translating $E$ by $(-1,-1)$ gives a lower-triangular diagram in
$\Delta_{n-2}$. By the induction hypothesis, increasing
all generator indices by two yields
$\delta(E^{\mathrm{sym}})=u^{-1}\circ u$.
Every generator occurring in $u$ has index at least $3$, so it commutes
with $s_1$. Thus
\begin{align*}
    \delta(D)^{-1}\circ\delta(D)
    =v^{-1}\circ u^{-1}\circ s_1^\epsilon
      \circ s_1^\epsilon\circ u\circ v =v^{-1}\circ s_1^\epsilon\circ u^{-1}\circ u\circ v =\delta(D^{\mathrm{sym}}),
\end{align*}
where the last equality is
\eqref{eq:appendix-completion-factorization}.
\end{proof}

Therefore, the set of symmetric completions of the involution Hecke pipe dreams is exactly the set of symmetric pipe dreams for the same involution. We now show that the unique maximal pipe dream introduced in~\cite{psw24} is symmetric.

\begin{lem}
\label{lem:appendix-symmetric-maximal-pipe-dream}
For $z\in\mI_n$, there exists a symmetric pipe dream
$P \in\PD(z)$ such that $|P|=\raj(z)$.
\end{lem}

\begin{proof}
Pechenik, Speyer, and Weigandt~\cite{psw24}*{Theorems~7.1 and~7.7}
showed that, for every $z\in S_n$, there exists a unique pipe dream
$P_z\in\PD(z)$ satisfying
\[
    \rwt(P_z)=\rajcode(z) \quad \text{and}\quad 
    \cwt(P_z)=\rajcode(z^{-1}).
\]
Since $z = z^{-1}$, the transpose $P_z^{\mathsf T}$ satisfies
\begin{align*}
    &P_z^{\mathsf T} \in \PD(z^{-1}) = \PD(z),\\
    &\rwt(P_z^{\mathsf T}) = \rajcode(z^{-1})= \rajcode(z), \text{ and}\\
    &\cwt(P_z^{\mathsf T}) = \rajcode(z)= \rajcode(z^{-1}).
\end{align*}
Therefore, by uniqueness, $P_z^{\mathsf T} = P_z$. Since $\rwt(P_z) = \rajcode(z)$, $|P_z| = \raj(z)$.
\end{proof}

We call $P_z$ the unique maximal pipe dream of $z$. There is an explicit construction of $P_z$ given by the author and Yu~\cite{cy24} for arbitrary permutations.

\begin{thm}
\label{thm:appendix-involution-degree}
For $z\in\mI_n$,
\[
    \deg\invG_z(\x)=\deg\fG_z(\x)=\raj(z).
\]
\end{thm}

\begin{proof}
For $D\in\mathcal D^{\inv}(z)$, let $d(D)$ and $o(D)$ be its numbers
of diagonal and strictly off-diagonal cells, so  $|D^{\mathrm{sym}}| = d(D) + 2 o(D)$. Then by~\eqref{eq:appendix-involution-pipe-dream} and 
Lemma~\ref{lem:appendix-symmetric-completion},
\[
    \deg\invG_z(\x)
    =\max_{D\in\mathcal D^{\inv}(z)}|D^{\mathrm{sym}}|
    =\max_{\substack{P\in\PD(z)\\P=P^{\mathsf T}}}|P|.
\]
Let $P_z$ be the unique maximal pipe dream of $z$. By Lemma~\ref{lem:appendix-symmetric-maximal-pipe-dream},~\eqref{eq:appendix-ordinary-pipe-dream}, and Theorem~\ref{t: psw raj},
\begin{align*}
     \raj(z) = |P_z| \leq \max_{\substack{P\in\PD(z)\\P=P^{\mathsf T}}}|P| \leq \max_{P\in\PD(z)}|P| = \deg\fG_z(\x) = \raj(z),
\end{align*}
and the statement follows.
\end{proof}

\bibliographystyle{alpha}
\bibliography{citation}
\end{document}